%% file: decouphighlowrev6.tex
\documentclass[11pt]{amsart}
\usepackage{amssymb}
\usepackage{amsthm}
\usepackage{verbatim}
\usepackage{hyperref}
\usepackage{color}
\usepackage{xcolor}
\usepackage{enumitem}

\newtheorem{prop}{Proposition}[section]
\newtheorem{lemma}[prop]{Lemma}

\newtheorem{theorem}[prop]{Theorem}

\newtheorem{note}[prop]{Note}

\numberwithin{equation}{section}

\newcommand{\RR}{\mathbb{R}}

\input{dfmacros}

\begin{document}

\title{Improved decoupling for the parabola}
\author{Larry Guth, Dominique Maldague, and Hong Wang}
\address{Department of Mathematics\\
Massachusetts Institute of Technology\\
Cambridge, MA 02142-4307, USA}
\email{dmal@mit.edu}

\date{\today}

\maketitle

\begin{abstract}
	We prove an $(\ell^2,L^6)$ decoupling inequality for the parabola with constant $(\log R)^{c}$. In the appendix, we present an application to the  sixth-order correlation of the integer solutions to $x^2+y^2=m$.
\end{abstract}

\tableofcontents

\section{Introduction and main results}

Let $f:\R^n\to\C$ be in the Schwartz class $\mathcal{S}$ with  Fourier support contained in $\mc{N}_{R^{-1}}(\mb{P}^{n-1})$, the $R^{-1}$ neighborhood of $\mathbb{P}^{n-1}:=\{(\xi, |\xi|^2), |\xi|\leq 1, \xi\in \mathbb{R}^{n-1}\}$. Let $\{\theta\}$ be a tiling of $\mc{N}_{R^{-1}}(\mb{P}^{n-1})$ by approximately  $R^{-1/2}\times \cdots \times R^{-1/2}\times R^{-1}$ rectangular boxes $\theta$ and define $f_{\theta}=(\widehat{f}\chi_{\theta})^{\vee}$. 

Let $D_{n,p}(R)$ denote the smallest constant such that 
\begin{equation}
\|f\|_{L^p(\mathbb{R}^n)} \leq D_{n, p}(R) (\sum_{\theta} \|f_{\theta}\|_{L^p(\mathbb{R}^n)}^2)^{1/2},
\end{equation}
for any $f\in \mathcal{S}$ with $\text{supp}\widehat{f}\subset \mathcal{N}_{R^{-1}}(\mathbb{P}^{n-1})$.

A trivial estimate using Cauchy-Schwarz and the triangle inequality yields $D_{n, p}(R) \leq R^{(n-1)/2}$.  And  we have $D_{n,p}(R)\geq1$ by taking $f=f_{\theta}$. Bourgain and Demeter \cite{BD14} proved that for $2\leq p\leq \frac{2(n+1)}{n-1}$, $D_{n, p}(R)\leq C_{\epsilon} R^{\epsilon}$ for any small $\epsilon>0$. Such estimates are possible due to the curvature of $\mathbb{P}^{n-1}$ and are sharp up to $R^{\epsilon}$--loss. The estimates have many applications in harmonic analysis, PDE and number theory.  It was conjectured that $D_{n,p}\leq C_p$ for $1\leq p< \frac{2(n+1)}{n-1}$. 

In this paper, we focus on the case $n=2$ and write $D_p(R)=D_{n,p}(R)$.  At the end point $p=6$, Bourgain proved in \cite{Bourgain93}  that $D_6(R)\gtrsim (\log R)^{1/6}$. Based on the Bourgain-Demeter decoupling, Zane Li  \cite{Zthesis}  proved that $D_6(R)\lesssim \exp( O(\frac{\log R \log \log \log R}{\log \log R}))$. Then, by adapting ideas from efficient congruencing, he proved \cite{ZLiEff} that  $D_6(R)\lesssim \exp( O(\frac{\log R}{\log \log R}))$.  This was the best previous bound for $D_6(R)$.  In this paper, we prove

\begin{theorem} \label{decoupling}
	$D_6(R)\lesssim (\log R)^{c'}$ for an absolute constant $c'$. 
\end{theorem}

Theorem~\ref{decoupling} is a corollary of our main theorem, which estimates the $L^6$--norm of $f$ on a subset of $\mathbb{R}^2$. 


%
%

\begin{theorem}\label{Main}  There exists $c>0$ such that the following holds. If $f\in\mc{S}$ has Fourier support contained in $\mc{N}_{R^{-1}}(\mb{P}^1)$ and $Q_R\subset\R^2$ is any cube of sidelength $R$, then
\[ \|f\|_{L^6(Q_R)}^6\le (\log R)^c (\sum_\theta\|f_\theta\|_{L^\infty(\mathbb{R}^2)}^2)^2\sum_\theta\|f_\theta\|_{L^2(\mathbb{R}^2)}^2\qquad\qquad\forall R\ge 2. \]
 \end{theorem} 

The proof of Theorem~\ref{Main}  is related to an incidence estimate between points and rectangles used in  \cite{DGW} and \cite{GWZ}.  These arguments are based on the following idea.  We consider the square function $g = \sum_\theta |f_\theta|^2$, and we divide it into a high-frequency part and a low frequency part.  For the high frequency part, the different terms $| f_\theta|^2$ are essentially orthogonal, and this gives a powerful tool when the high frequency part of $g$ dominates.  When the low frequency part of $g$ dominates, we try to reduce the whole problem to a similar problem at a coarser scale.  

We use these tools to give a different proof of decoupling for the parabola.  Compared to the two previous proofs (by Bourgain-Demeter \cite{BD14} and Li \cite{ZLiEff}, our proof leans less heavily on induction on scales, and we think this is the main reason it gives a stronger estimate.  In order to obtain the $(\log R)^{c'}$ bound, we also need to deal carefully with a number of technical difficulties. These include  a wave packet decomposition using Gaussian partitioning of unity, carefully modifying the function at each scale and reducing to a well-spaced frequency case. We give an intuitive explanation of the argument in Section 1. 
 
 The bound $(\log R)^{c'}$ is useful compared to $R^{\epsilon}$ in some diophantine equation problems. Let $\Lambda_m =\{ (x, y)\in \mathbb{Z}^2, x^2+y^2=m\}$.  In \cite{BB15}, Bombieri and Bourgain studied the number of solutions of the system $\lambda_1+\lambda_2+\lambda_3=\lambda_4+\lambda_5+\lambda_6$ with $\lambda_j \in \Lambda_m$. In \cite{ZLi}, Li and Bourgain applied decoupling to this problem.  They were able to prove a very strong bound for the number of solutions provided that $\Lambda_m$ is very large.  Using our stronger estimate for $D_6(R)$, we can extend their bound to a wider range of $\Lambda_m$.   We present this application in the appendix.  
  
 Another corollary of Theorem~\ref{decoupling} concerns the discrete Fourier restriction on $\{(n, n^2), n\in \mathbb{Z}\}$. 
 \begin{corollary}Let $K_p(N)$ denote the smallest constant such that for any $\{a_n\}_{|n|\leq N}$, 
 	$$\|\sum_{|n|\leq N} a_n e^{2\pi i (nx+n^2t)} \|_{L^p(\mathbb{T}^2)}\lesssim K_p(N) (\sum_{|n|\leq N} |a_n|^2)^{1/2}.$$
 	Then $K_6(N)\lesssim (\log N)^{c'}$. 
 \end{corollary}
 
 \noindent Bourgain showed in  \cite{Bourgain93} ((2.51), Proposition 2.36) that $ c(\log N)^{1/6} \leq K_6(N) \leq \exp (c \frac{\log N}{\log \log N})  $.   He also asked whether $K_p(N)$ is bounded independent of $N$ for each $p < 6$.  
  
%

\section{Intuitive explanation of the argument \label{intuition} }

In this section, we outline the main ideas of the proof.  For simplicity, we suppress some minor technical details, but at the end we will discuss the most important technical issues that come up.  

Let $\text{Dec}(R)$ be the optimal constant in the decoupling  inequality
\[  \|f\|_{L^6(Q_R)}\le\text{Dec}(R) (\sum_\theta\|f_\theta\|_{L^6(\w_R)}^2)^{1/2}  . \]
Here the $\theta$ denote $\sim R^{-1/2}\times R^{-1}$ approximate rectangles which partition an $R^{-1}$ neighborhood of $\mb{P}^1$ and $\widehat{f_\theta}=\widehat{f} \chi_\theta$ where $f$ is a Schwartz function. The original arguments of Bourgain and Demeter to prove that $\text{Dec}(R)\le C_\e R^\e$ involve analysis of $f_{\tau}$ where $\tau$ is a rectangle in a neighborhood of $\mb{P}^1$ and $\tau$ is at various scales between $1$ and the final scale $R$. This is also true for the proof of Theorem \ref{Main}, which involves analysis of $f$ at $\sim \log R$ many scales. 
Use the notations $A\lesssim B$ and $A\lessapprox B$ to mean $A\le CB$ and $A\le (\log R)^CB$, respectively, for some absolute constant $C$.

By a standard pigeonholing argument (see \textsection\ref{pigeonholing}), the decoupling inequality above follows from the estimate
\begin{equation} \label{mainest} \a^6|\{x\in Q_R:|f|\sim\a\}|\lessapprox \big(\sum_\theta\|f_\theta\|_{L^\infty(Q_R)}^2\big)^2\sum_\theta\|f_\theta\|_{L^2(Q_R)}^2  \end{equation}

\noindent where we may assume that for each $\theta$,

\begin{equation} \label{pigeonintuit} \|f_\theta\|_{L^\infty(Q_R)}\sim 1 \textrm{ or } f_\theta = 0 \end{equation}

\noindent and that $\|f_\theta\|_{L^p(Q_R)}$ are comparable for all non-zero $f_\theta$ and all $2 \le p \le 6$.  Note that inequality (\ref{mainest}) is also (roughly) the statement of Theorem \ref{Main}. In this section, we are suppressing the weight functions localized to $Q_R$ which are present in the $L^2$-norms on the right hand side above. 

Recall the reverse square function estimate for $L^4$, which says that
\begin{equation}\label{revsq}  
\a^4|\{x\in Q_R:|f|\sim\a\}|\lesssim \int_{Q_R}\big(\sum_\theta|f_\theta|^2\big)^2 . \end{equation}
If we use Minkowski's integral inequality to interchange the $\ell^2$ and $L^4$ norms, then we immediately get $L^4$-decoupling
\begin{equation}\label{l2L4}  \a^4|\{x\in Q_R:|f|\sim\a\}|\lesssim \Big(\sum_\theta\big(\int_{Q_R}|f_\theta|^4\big)^{1/2} \Big)^{2}. \end{equation}
The heart of our argument involves analyzing special cases where we can upgrade (\ref{revsq}) into something that implies $L^6$-decoupling. We will describe the simplest special case of the argument now. 

\vspace*{2mm}
\noindent{\bf{Special case: High frequency dominance.}} 
\vspace{2mm}
Consider the square function that appears on the right hand side of (\ref{revsq}), $\sum_\theta |f_\theta|^2$. Each summand $|f_\theta|^2=f_\theta\overline{f_\theta}$ has Fourier support in $\theta-\theta$, which looks like a copy of the $R^{-1/2}\times R^{-1}$ rectangle $\theta$ that is dilated by a factor of $2$ and translated to the origin. Let $\eta$ be a smooth approximation of the characteristic function of a ball of radius $R^{-1/2}/\log R$. Define the low frequency part as 
\[ \big(\sum_\theta|f_\theta|^2\big)_\ell:=\sum_\theta|f_\theta|^2*\widecheck{\eta}\]
and the high part by
\[ \big(\sum_\theta|f_\theta|^2\big)_h:=\sum_\theta|f_\theta|^2-\big(\sum_\theta|f_\theta|^2\big)_\ell.  \]

In this special case, we assume that
\[ \int_{Q_R}\big(\sum_\theta|f_\theta|^2\big)^2\lessapprox \int_{Q_R}\big|\big(\sum_\theta|f_\theta|^2\big)_h\big|^2. \]
We would like to use Plancherel's theorem to analyze the right hand side, but we are integrating over $Q_R$ instead of $\R^2$. This is solved by using a weight function $w_{Q_R}$ to approximate $\chi_{Q_R}$ that has the property that $\widehat{w}_{Q_R}$ is supported in a ball of radius $R^{-1}$. This blurs the support properties of the Fourier transform of $\big(\sum_\theta|f_\theta|^2\big)_h$ by a factor of $R^{-1}$, but does not alter the main properties. In particular, $\big(\sum_\theta|f_\theta|^2\big)_h$ has Fourier transform supported on $\underset{\theta}{\cup}(\theta-\theta)$ intersected with the compliment of the ball centered at the origin of radius $R^{-1/2}/\log R$. Since the $\theta$ cover a small  $R^{-1}$-neighborhood of $\mb{P}^1$ (which curves), the rectangles $\theta-\theta$ are oriented at angles which are pairwise $\gtrsim R^{-1/2}$-separated. The overlap of these rectangles outside of $B(R^{-1/2}/\log R)$ is $\sim \log R$. Thus, by Cauchy-Schwarz, 
\[ \int_{Q_R}\big|\big(\sum_\theta|f_\theta|^2\big)_h\big|^2\lesssim \log R\int_{Q_R} \sum_\theta|f_\theta|^4 \]
where technically the auxiliary function $\eta$ associated to the definition of the high and low frequency parts is absorbed into a weight function discussed in more detail below. 

To summarize, we have so far that
\[ \a^4|\{x\in Q_R:|f|\sim \a\}|\lessapprox \int_{Q_R}\sum_\theta|f_\theta|^4. \]
This means that we have upgraded our $(\ell^2,L^4)$-decoupling result (\ref{l2L4}) into an $(\ell^4,L^4)$-decoupling result. 
Note that for some $x\in Q_R$, $\a\sim|f(x)|\le\underset{\theta}{\sum}\|f_\theta\|_{L^\infty(Q_R)} \sim \underset{\theta}{\sum}\|f_\theta\|_{L^\infty(Q_R)}^2$ and each $\|f_\theta\|_\infty\lesssim 1$. Thus,  
\[ \a^4|\{x\in Q_R:|f|\sim\a\}|\lessapprox \a^{-2}\big(\underset{\theta}{\sum}\|f_\theta\|_{L^\infty(Q_R)}^2\big)^2\sum_\theta\int_{Q_R}|f_\theta|^2 \]
which is the $L^6$-decoupling result we were aiming for and concludes the special case. 

\qed

Suppose that we are not in the high frequency dominating case above, but that we have a high frequency dominance at a different scale $\tilde{R}$: 
\[ \int_{Q_R}\big(\sum_\tau|f_\tau|^2\big)^2\lessapprox \int_{Q_R}\big|\big(\sum_\tau|f_\tau|^2\big)_h\big|^2  .      \]
where the $\tau$ are $\tilde{R}^{-1/2}\times \tilde{R}^{-1}$ rectangles covering a $\tilde R^{-1}$-neighborhood of $\mb{P}^1$ and the high part is with respect to this new scale. If we repeat the above argument, we obtain the $(\ell^4,L^4)$ result
\[ \a^4|\{x\in Q_R:|f|\sim\a\}|\lessapprox \int_{Q_R}\sum_\tau|f_{\tau}|^4. \]
If we try to relate the right hand side to a sum of $L^2$-norms, then
\[ \a^4|\{x\in Q_R:|f|\sim\a\}|\lessapprox \big(\max_{\tau}\|f_\tau\|_\infty^2\big)\sum_\tau\int_{Q_R}|f_{\tau}|^2.  \]
By $L^2$-orthogonality, this is equivalent to
\[ \a^4|\{x\in Q_R:|f|\sim\a\}|\lessapprox \big(\max_{\tau}\|f_\tau\|_\infty^2\big)\sum_\theta\int_{Q_R}|f_{\theta}|^2. \]
The issue now is that in the special case above, we had good control over$\|f_\theta\|_\infty$ given in (\ref{pigeonintuit}), but we don't have any corresponding estimate for  $\|f_\tau\|_\infty$. A key part of the proof is a pruning process for the wave packets of $f_{\tau}$ which will allow us to control $\|f_{\tau}\|_\infty$.  

Our argument will involve many scales, and so we introduce a sequence of intermediate scales and high-low decompositions for each scale. 
Denote the intermediate scales by
\[ 1<R_1<\cdots<R_k<R_{k+1}<\cdots<R_N=R . \]
We will use scales which have the property that
\begin{equation}\label{eq: ratioRk}
 \frac{R_{k+1}}{R_k}\sim (\log R)^c. 
 \end{equation}
Let $\{\tau_k\}$ denote $R_k^{-1/2}\times R_k^{-1}$ rectangles which partition an $R_k^{-1}$ neighborhood of $\mb{P}^1$. Note that for each $k=1,\ldots,N$, 
\begin{equation}\label{eq: decomftauk}
 f=\sum_{\tau_k}f_{\tau_k}. 
\end{equation}
We analyze the square functions
\[ g_k=\sum_{\tau_k}|f_{\tau_k}|^2 .\]
Intuitively since the first scale $R_1\approx 1$, we have by Cauchy-Schwarz that
\[ |f|\lessapprox g_1^{1/2}, \]
and we should think of $g_1$ as being close to $|f|^2$. On the other hand, $g_N$ is our original square function $\sum_\theta |f_\theta|^2$.  

We define a high-low decomposition for $g_k$, building on \cite{DGW} and \cite{GWZ}. As in the special case above, observe that the Fourier transform of $g_k$ is 
\[ \widehat{g_k}=\sum_{\tau_k}\widehat{|f_{\tau_k}|^2}=\sum_{\tau_k}\widehat{f}_{\tau_k}*\widehat{\overline{f}}_{\tau_k}. \]
By definition, $\widehat{f}_{\tau_k}$ has support on $\tau_k$ and $\widehat{\overline{f}}_{\tau_k}$ has support on $-\tau_k$. Thus 
\[  \text{supp} \widehat{|f_{\tau_k}|^2}\subset\tau_k-\tau_k \]
and $\tau_k-\tau_k$ is the same as $\tau_k$ translated to the origin and dilated by a factor of $2$. Since the $\{\tau_k\}$ formed a partition of the neighborhood of the parabola, the Fourier support of $g_k$ is a union of $\sim R_k^{-1/2}\times R_k^{-1}$ rectangles centered at the origin oriented at $\sim R_k^{-1/2}$-separated angles. Note that the intersection of all of these tubes is an $R_k^{-1}$-ball centered at the origin, and outside of some neighborhood of the origin, the tubes look more disjoint (or at least finitely overlapping). This is the setting for a high-low frequency decomposition. 

We separate out a low-frequency part of $g_k$ and a high frequency part of $g_k$.
\begin{definition}\label{def: bumpetak}
  Let $\eta_k(\xi)$ be a bump function 
associated to a ball of radius $\rho_k,$ where   $\rho_k= (\log R)^{-c} R_k^{-1/2}$ with $c$ in \eqref{eq: ratioRk}.  We have $\eta_k(\xi) = 1$ on $B_{\rho_k}$ and $\eta_k(\xi)=0$ outside of $2B_{\rho_k}$. 
\end{definition}
  Define the low frequency part of $g_k$ by  
\[ \widehat{g_{k,\ell}}=\eta_k \widehat{g_k} \]

Let the high part of $g_k$ be equal to 
\[  g_{k,h}=g_k-g_{k,\ell}. \] 

The following lemmas describe the good features of the low-frequency part and the high-frequency part of $g_k$.  

\begin{lemma}[``Low lemma"] Given $\rho_k =(\log R)^{-c} R_k^{-1/2} \le R_{k+1}^{-1/2}$,
\[ |g_{k,\ell}(x)|\le C_{low}\, g_{k+1}*|\widecheck{\eta}_k|(x) \]
\end{lemma}

\begin{proof} [Sketch] First write $g_{k,\ell}$ 
\begin{align*}
    g_{k,\ell}(x)&=\sum_{\tau_k}\int|f_{\tau_k}|^2(y)\widecheck{\eta}_{k}(x-y)dy .
\end{align*} 
Each $\tau_k$ is a union of $R_{k+1}^{-1/2}\times R_{k+1}^{-1}$ rectangles, so we write $f_{\tau_k}=\sum\limits_{\tau_{k+1}\subset\tau_k} f_{\tau_{k+1}}$. Morally, the $f_{\tau_{k+1}}$ are orthogonal in $L^2$ on balls of radius $\gtrsim R_{k+1}^{1/2}$. Since $|\widecheck{\eta}_{k}|$ is locally constant on balls of radius $\sim \rho_k^{-1} \gtrsim R_{k+1}^{1/2}$, we invoke the local orthogonality of the $f_{\tau_{k+1}}$ to conclude 
\[ |g_{k,\ell}(x)|\lesssim \sum_{\tau_k}\sum_{\tau_{k+1}\subset\tau_k}\int|f_{\tau_{k+1}}|^2(y)|\widecheck{\eta}_k|(x-y)dy. \]
\end{proof}

Since $g_{k+1}$ is locally constant on balls of radius $\sim R_{k+1}^{1/2}$, we may also update the low lemma above to say

\begin{corollary}[White lie! ``Low corollary"]
	\[ g_{k,\ell}(x)\le C_{low}\,g_{k+1}(x). \]
\end{corollary}

\begin{lemma}[``High lemma"] \label{highoutline} For any $B_{R_k}$ ball, 
\[ \int_{B_{R_k}}|g_{k,h}|^2\lesssim \rho_k^{-1} R_k^{-1/2}\sum_{\tau_k}\int|f_{\tau_k}|^4\w_{R_{k+1}}     \] 
where $1_{B_{R_{k+1}}}\le \w_{R_{k+1}}$ and $\widehat{\w}_{R_{k+1}}$ is supported in a ball of radius $2R_{k+1}^{-1}$. 
\end{lemma}

Remark. Since $\rho_k = (\log R)^{-c} R_k^{-1/2}$, and so the factor $\rho_k^{-1} R_k^{-1/2}$ is bounded by $(\log R)^c$.  

\begin{proof}[Sketch of ``High Lemma"]
\begin{align*}
    \int_{B_{R_{k+1}}}|g_{k,h}|^2&\le \int|g_{k,h}|^2\w_{R_{k+1}}\\
    &=\int\widehat{g}_{k,h}\,(\,\widehat{\overline{g}}_{k,h}*\widehat{\w}_{R_{k+1}}). 
\end{align*}     
The Fourier transform of $g_{k,h}$ is supported outside of a ball centered at the origin of radius $\rho_k$. The support of $\widehat{\overline{g}}_{h,k}*\widehat{\w}_{R_{k+1}}$ is outside of a ball of radius $\rho_k -R_{k+1}^{-1}\sim \rho_k$. This means that the previous displayed math equals
\[ \int_{|\xi|\gtrsim \rho_k} \widehat{g}_{k,h}\,(\,\widehat{\overline{g}}_{k,h}*\widehat{\w}_{R_{k+1}}) \sim\sum_{\tau_k}\sum_{\tau_k'}\int_{|\xi|\gtrsim \rho_k} \widehat{|f_{\tau_k}|^2}\,(\,\widehat{\overline{|f_{\tau_k'}|^2}}*\widehat{\w}_{R_{k+1}}). \]

On the region $|\xi|\gtrsim \rho_k$, the support of $\widehat{|f_{\tau_k}|^2}$ intersects the support of $|\widehat{f_{\tau_k'}|^2}*\widehat{\w}_{R_{k+1}}$ for at most $\rho_k^{-1} R_k^{-1/2}$ many $\tau_k'$. It follows from Plancherel's theorem and Cauchy-Schwarz that
\[ \sum_{\tau_k}\sum_{\tau_k'}\int_{|\xi|\gtrsim \rho_k}\widehat{|f_{\tau_k}|^2}\,(\,\widehat{|f_{\tau_k'}|^2}*\widehat{\w}_{R_{k+1}})\lesssim \rho_k^{-1} R_k^{-1/2}\sum_{\tau_k}\int|f_{\tau_k}|^4\w_{R_{k+1}}. \]

\end{proof}

The ``low lemma'' tells us that

$$ g_k = g_{k, \ell} + g_{k, h} \le C_{low} g_{k+1} + |g_{k,h}|. $$

Therefore, either $g_k(x) \le A |g_{k,h}|$ or $g_k(x) \le \frac{A}{A-1} C_{low} g_{k+1}(x)$.   Here $A$ is a parameter that we can choose later. This leads to a partition of the domain into the following sets:
\[ Q_R=L\sqcup \Omega_1\sqcup\cdots\sqcup \Omega_{N-1}. \]
Define 
\begin{align*} 
\Omega_{N-1}=\{x\in Q_R:g_{N-1}(x)\le A |g_{N-1,h}(x)|\}. 
\end{align*} For $k=1,\ldots,N-2$, define
\begin{align*}
\Omega_k&=\{x\in Q_R\setminus(\Omega_{k+1}\cup\cdots\cup\Omega_{N-1}):g_{k}(x)\le A |g_{k,h}(x)|\}\\
&\subseteq \{x\in Q_R:g_k(x)\le A |g_{k,h}(x)|, g_{\ell} \le \frac{A}{A-1}  C_{low} g_{\ell+1} \textrm{ for $k+1 \le \ell \le N-1$}\}.
\end{align*}

The set $L$ is defined as 
\[ Q_R\setminus(\Omega_1\sqcup\cdots\sqcup \Omega_{N-1})\subseteq \{x\in Q_R:g_{\ell} \le \frac{A}{A-1}  C_{low} g_{\ell+1} \textrm{ for $1 \le \ell \le N-1$}\}.
 \]
Since we have partitioned $Q_R$ into $\sim\log R$ many sets, it suffices to consider the cases 
\begin{equation}\label{eq: lowdominate}
\|f\|_{L^6(Q_R)}\lesssim (\log R)\|f\|_{L^6(L)}
\end{equation}
or for some $k$,
\begin{equation}\label{eq: highdominate}
\|f\|_{L^6(Q_R)}\lesssim (\log R)\|f\|_{L^6(\Omega_k)}. 
\end{equation}

The case where $L$ dominates, which means \eqref{eq: lowdominate} holds,  is simple because for $x \in L$, $|f(x)|^2 \lesssim g_1(x) \lesssim g_N (x) = \sum_{\theta} |f_\theta(x)|^2$, and so

\[ \int_{L}|f|^6\lesssim \int_L(\sum_\theta|f_\theta|^2)^3\le (\sum_\theta\|f_\theta\|_{L^\infty(Q_R)}^2)^2\sum_\theta\|f_\theta\|_{L^2(Q_R)}^2,\]
which implies the conclusion of Theorem \ref{Main}. 

More quantitatively, for $x \in L$, we have the bound

$$|f(x)|^2 \lesssim g_1 \le \left( \frac{A}{A-1} C_{low} \right)^N \sum_{\theta} |f_\theta(x)|^2. $$

\noindent This ultimately gives us a bound for $\text{Dec}(R)$ of the form $\left( \frac{A}{A-1} C_{low} \right)^N$.  Recall that $N \sim \log R / \log \log R$.  We can choose $A$ to be large to control the contribution of $\frac{A}{A-1}$.  But if $C_{low}$ is a constant bigger than 1, then $C_{low}^N$ will be much larger than $(\log R)^C$ (although still smaller than $C_\eps R^\eps$).  We will have to work more carefully in the low lemma to make $C_{low}$ very close to 1.  We will return to discuss this more below.

But first we discuss the case where one of the $\Omega_k$ dominates.  In this case we begin by applying a broad/narrow analysis.  The narrow case is handled by an induction on scales argument.  For the broad case, we consider the set

\[ U:=\{x\in Q_R:|f(x)|\sim \a,\quad|f(x)|\lessapprox \max_{\tau_1,\tau_1' \textrm{ non-adjacent}}|f_{\tau_1}f_{\tau_1'}|^{1/2}(x)\}\]

\noindent where $\tau_1$ and $\tau_1'$ are $R_1^{-1} \times R_1^{-1/2}$ rectangles.  

The $L^\infty$ norm of $g_N$ plays an important role in our argument, so we give it a name:

$$ r = \| g_N \|_{L^\infty(Q_R)}. $$

\noindent Since $r = \| \sum_{\theta} |f_\theta|^2 \|_{L^\infty(Q_R)} \le \sum_\theta \| f_\theta \|_{L^\infty(Q_R)}^2$, the main estimate (\ref{mainest}) follows from the bound

\begin{equation} \label{goal} \alpha^6 | U \cap \Omega_k | \lessapprox r^2 \sum_\theta \| f_\theta \|_{L^2(Q_R)}^2. \end{equation}

Focusing on the broad case allows us to use bilinear restriction, which leads to the following bound:

$$ \a^4 | U \cap \Omega_k | \lessapprox \int_{\Omega_k} g_k^2. $$

\noindent From here, use that $g_k\lessapprox |g_{k,h}|$ on $\Omega_k$ and proceed as in the special case above, using the ``high lemma", Lemma \ref{highoutline}, to obtain 
\[  \a^4|U\cap\Omega_k|\lessapprox \max_{\tau_k}\|f_{\tau_k}\|_{L^\infty(Q_R)}^2\sum_\theta\|f_\theta\|_{L^2(Q_R)}^2. \]
Recall that we must modify the function $f$ to get a good bound for the $\|f_{\tau_k}\|_{L^\infty(Q_R)}$. Here is the idea for modifying $f$.  If $x \in U$, then we know that $|f(x)| \sim \a$.  By the definition of $\Omega_k$, we know that for $x \in \Omega_k$,

\begin{equation} \label{eq: startpruning}
\sum_{\tau_k} |f_{\tau_k}(x)|^2 \le K r, 
\end{equation}

\noindent where $K=(\log R)^c (\frac{A}{A-1}C_{low})^{N-k-1}. $  Even though $K$ might be as large as $C^{\frac{\log R}{\log \log R}}$, this is accounted for in the proof with a more careful definition of $\Omega_k$. 
Inequality~\eqref{eq: startpruning} implies that for $x \in U \cap \Omega_k$, the $f_{\tau_k}$ with $|f_{\tau_k}(x)| > 100 K r / \a$ make a small contribution to $f(x)$.  More precisely, if $x \in U \cap \Omega_k$, then

\begin{equation} \label{highamplneg}    \sum_{\tau_k: |f_{\tau_k}(x)| > 100 K r / \a } | f_{\tau_k}(x)| \le \frac{\a}{100 K r} \sum_{\tau_k} |f_{\tau_k}(x)|^2 \le \frac{\a}{100} \le \frac{1}{10} |f(x)|. \end{equation}

\noindent We define 
\begin{equation}\label{def: lambda}
\lambda = 100 K r \ /a.
\end{equation}
  Roughly speaking, the parts of $f_\tau$ with norm bigger than $\lambda$ don't make a significant contribution to $f$ on the set $U \cap \Omega_k$.  To take advantage of this observation, we divide each $f_{\tau_k}$ into wave packets, and then prune the wave packets with amplitude bigger than $\lambda$.  

The pruning process goes roughly as follows (but this account is a little oversimplified).  First we expand $f_{\tau_k}$ into wave packets,
\begin{equation}\label{eq: babywp}
f_{\tau_k} = \sum_{T} \psi_T f_{\tau_k}.
\end{equation}
\noindent Here $T$ denotes a translate of the dual convex $\tau_k^*$  (see Definition~\ref{def: dual convex}), and the sum is over a collection of translates that tile the plane.  The function $\psi_T$ is a smooth approximation of the characteristic function of $T$, and the $\psi_T$ form a partition of unity.  Each $\psi_T f_{\tau_k}$ is called a wave packet, and it has Fourier support essentially contained in $\tau_k$.  We define $\tilde f_{\tau_k}$ be the result of pruning the high amplitude wave packets from $f_{\tau_k}$:

$$ \tilde f_{\tau_k} = \sum_{T: \| \psi_T f_{\tau_k} \|_\infty \le \lambda} \psi_T f_{\tau_k}. $$

\noindent The Fourier support of $\tilde f_{\tau_k}$ is still essentially contained in $\tau_k$.   Suppose for a moment that $\psi_T$ was just $\chi_T$, the characteristic function of $T$.  Then because of our pruning, $\| \tilde f_{\tau_k} \|_{\infty} \le \lambda$.   Next we define $f_k = \sum_{\tau_k} \tilde f_{\tau_k}$.

To analyze $|U \cap \Omega_k|$, we use the argument above with $f_k$ in place of $f$ and $\tilde f_{\tau_k}$ in place of $f_{\tau_k}$.  There are the key features of $f_k$ that makes this possible:

\begin{itemize}

\item The function $f_k$ is close to $f$ on $U \cap \Omega_k$.  If $\psi_T$ was just $\chi_T$, then the analysis in (\ref{highamplneg}) would show that for $x \in U \cap \Omega_k$, $|f(x) - f_k(x)| \le \frac{1}{100} \a$.  We will ultimately define $f_k$ in a slightly more complicated way, and we will prove this bound for $|f(x) - f_k(x)|$.  

\item We now have the bound $\| \tilde f_{\tau_k} \|_{L^\infty} \le \lambda \approx r / \a$.

\item The function $f_k$ has Fourier support properties similar to those of $f$ so that we can run the argument above.  For instance, the Fourier support of $\tilde f_{\tau_k}$ is essentially contained in $\tau_k$.

\end{itemize}

When we run the argument above with $f_k$ in place of $f$, and then plug in the bound $\| \tilde f_{\tau_k} \|_{L^\infty} \lessapprox r / \a$, we get the estimate

\begin{equation} \label{afterhigh} \a^4 | U \cap \Omega_k | \lessapprox (r/\a)^2 \int_{Q_R}  \sum_{\theta} | f_{\theta}|^2. \end{equation}

\noindent This is our desired estimate (\ref{goal}).  

\subsection{Technical issues}

There are two main sources of technical difficulties that come up in implementing the sketch above.  One has to do with proving the low lemma with very sharp control.  The other has to do with pruning wave packets, which we have to do at many different scales. To make the argument work rigorously, $g_k$ and $f_k$ both have to be defined in a more complex way than they were above.

In order to obtain the bound $\text{Dec}(R) \le (\log R)^c$, we also have to make many small strategic decisions in our argument.  The guiding principle behind all these decisions is the following: to obtain a $(\log R)^c$ bound with an argument involving $N\sim \log R/\log\log R$ steps, we must carefully control any step which involves $N$ iterations. On the other hand, for steps that are only iterated $O(1)$ times, we can prove lossier bounds.

A good example is the low lemma.  If we are not very careful with how we formulate the ``low lemma'', we will get a bound for $\text{Dec}(R)$ which is much larger than $(\log R)^c$.  As we discussed above, if we prove the low lemma in the form $| g_{k, \ell}(x)| \le C_{low} g_{k+1}(x)$, then we will get a bound for $\text{Dec}(R)$ which is at least as big as $C_{low}^N$.  To get our desired bound for $\text{Dec}(R)$, we need $C_{low}$ to be almost 1.

Above we gave a non-rigorous sketch of the low lemma.  To get some perspective, let us now rigorously prove a version of the low lemma to get a perspective on $C_{low}$.

\begin{lemma}[Baby low lemma] \label{babylow} 
Let $\eta_k(\xi)$ be a bump function defined as in Definition~\ref{def: bumpetak}.  Then
$$ |g_{k, \ell}| = |\sum_{\tau_k} | f_{\tau_k}|^2 * \widecheck{\eta}_k | \le 2 \sum_{\tau_{k+1}} |f_{\tau_{k+1}}|^2 * |\widecheck{\eta}_k|. $$
\end{lemma}

\begin{proof} We write $\sum_{\tau_k} | f_{\tau_k}|^2 * \widecheck{\eta}_k(x)$ using Fourier inversion.

$$ \sum_{\tau_k} | f_{\tau_k}|^2 * \widecheck{\eta}_k(x) = \sum_{\tau_k} \int \widehat{f}_{\tau_k} * \widehat{ \overline{f}}_{\tau_k} e^{2 \pi i \xi \cdot x} \eta_k(\xi) d \xi. $$

Now $f_{\tau_k} = \sum_{\tau_{k+1} \subset \tau_k} f_{\tau_{k+1}}$, so we can expand out the last expression:

$$ = \sum_{\tau_k} \sum_{\tau_{k+1}, \tau_{k+1}' \subset \tau_k} \int \widehat{f}_{\tau_{k+1}} * \widehat{ \overline{f}}_{\tau_{k+1}'} e^{2 \pi i \xi \cdot x} \eta_k(\xi) d \xi. $$

Now the point is that most of the integrals in the sum above vanish.  The convolution $\widehat{f}_{\tau_{k+1}} * \widehat{ \overline{f}}_{\tau_{k+1}'}$ is supported in $\tau_{k+1} - \tau'_{k+1}$, and $\eta_k$ is  supported in a ball of radius $2\rho_k \le R_{k+1}^{-1/2}$.  Now each rectangle $\tau_{k+1}$ has dimensions $R_{k+1}^{-1} \times R_{k+1}^{-1/2}$.  So $\tau_{k+1} - \tau_{k+1}'$ intersects the support of $\eta_k$ only if $\tau_{k+1}'$ is equal to or adjacent to $\tau_{k+1}$.  We keep only these terms in the sum to get

$$ \sum_{\tau_k} |f_{\tau_k}|^2 * \widecheck{\eta}_k = \sum_{\tau_{k+1}, \tau_{k+1}' \text{equal or adjacent} } (f_{\tau_{k+1}} \overline{f_{\tau_{k+1}'}} )* \widecheck{\eta}_k. $$

For the cross terms, we note that
$$  |(f_{\tau_{k+1}} \overline{f_{\tau_{k+1}'}} ) * \widecheck{\eta}_k|  \le  (|f_{\tau_{k+1}}| |f_{\tau_{k+1}'}|) *  |\widecheck{\eta}_k| \le \left(  \frac{1}{2} |f_{\tau_{k+1}}|^2 + \frac{1}{2} |f_{\tau_{k+1}'}|^2  \right) *  |\widecheck{\eta}_k|.$$

Finally, grouping all the terms gives the desired bound:

$$ |\sum_{\tau_k} | f_{\tau_k}|^2 * \widecheck{\eta}_k | \le 2 \sum_{\tau_{k+1}} |f_{\tau_{k+1}}|^2 * |\widecheck{\eta}_k|. $$

\end{proof}

There are a couple of issues with this bound.  One issue is that we have an unwanted factor of 2 on the right-hand side.  A second issue is that we have a convolution on the right-hand side.  If we take $g_k = \sum_{\tau_k} | f_{\tau_k}|^2$, then we have $|g_{k, \ell}| \le 2 g_{k+1} * |\widecheck{\eta}_k|$.

To deal with the factor of 2, we consider a special case when the Fourier support of $f$ has a helpful spacing condition.  
Let $\Theta=\{\theta\}$ be a collection of $\sim R^{-1/2}\times R^{-1}$ rectangles contained in an $R^{-1}$-neighborhood of $\mb{P}^1$. The collection $\Theta$ has the spacing property at scale $R_k$ if there exists a collection of $\sim R_k^{-1/2}\times R_k^{-1}$-rectangles $\tau_k$ which cover $\cup_{\theta\in\Theta} \theta$ and such that
\[\text{dist}(\tau_k,\tau_k')\ge (\log R)^{-1} R_{k}^{-1/2} \]
whenever $\tau_k$ and $\tau_k'$ are distinct. If $\Theta$ has the spacing property at scales $R_1,\ldots,R_{N-1}$, then say $\Theta$ is well-spaced.  A well-spaced collection of rectangles $\theta$ can include most of the rectangles needed to cover the parabola, and we will be able to reduce our theorem for a general $f$ to the case that the Fourier support of $f$ is well-spaced.  The spacing condition helps us because whenever $\tau_{k+1}$, $\tau_{k+1}'$ are distinct, $\tau_{k+1} - \tau_{k+1}'$ is supported outside the ball of radius $(\log R)^{-1} R_{k+1}^{-1/2}$.  Now we choose $\rho_k \le (\log R)^{-1} R_{k+1}^{-1/2}$, and we see that all the cross terms in Lemma \ref{babylow} vanish.  This gets rid of the factor of 2.   Assuming that $f$ obeys the spacing condition, we conclude that 

$$ |\sum_{\tau_k} | f_{\tau_k}|^2 * \widecheck{\eta}_k | \le \sum_{\tau_{k+1}} |f_{\tau_{k+1}}|^2 * |\widecheck{\eta}_k|. $$

Next we discuss the $* | \widecheck{\eta}_k|$ on the right-hand side.  In order to deal with this factor, we define $g_k$ in a more complicated way.  Our definition of $g_k$ has the following form:

\begin{equation}\label{eq: babypsik}
 g_k:=\sum_{\tau_k}|f_{k+1,\tau_k}|^2*\p_{\tilde{T}_{\tau_k}}.
\end{equation}

Here $f_{k+1, \tau_k}$ is given by pruning high amplitude wave packets from $f_{\tau_k}$, and we will discuss it more below.  The function $\p_{\tilde{T}_{\tau_k}}$ is roughly $ \frac{1}{| \tau_k^*|} \chi_{\tau_k^*}$.  It's a bit bigger than this, so a more accurate model is 

$$ \p_{\tilde{T}_{\tau_k}} \sim \frac{(\log R)^c}{|\tau_k^*|} \chi_{(\log R)^c \tau_k^*}. $$

Let's see why this extra convolution helps us.  In the well-spaced case, the argument above shows that

\begin{equation} \label{lowwithnewgk} | g_k * \widecheck{\eta_k} | \le \sum_{\tau_{k+1}} |f_{k+1, \tau_{k+1}}|^2 * \p_{\tilde{T}_{\tau_k}} * |\widecheck{\eta_k}|. \end{equation}

\noindent On the right-hand side, $\tau_k$ denotes the parent of $\tau_{k+1}$.  We choose the functions $\p_{\tilde{T}_{\tau_k}}$ so that for any $\tau_{k+1} \subset \tau_k$, 

\begin{equation} \label{precur} \p_{\tilde{T}_{\tau_k}} * |\widecheck{\eta_k}| \le \p_{\tilde{T}_{\tau_{k+1}}}. \end{equation}

\noindent With this choice, the right hand side of (\ref{lowwithnewgk}) is bounded by $ \sum_{\tau_{k+1}} |f_{\tau_{k+1}}|^2 * \p_{\tilde{T}_{\tau_{k+1}}} \le g_{k+1}$.  So with this definition, we get $|g_{k, \ell}| \le g_{k+1}$.  (We will prove this in Lemma \ref{low}.)

Redefining $g_k$ in this way makes the statement of the low lemma very clean.  It does have a cost though.  We have to make sure that the contribution of $\p_{\tilde{T}_{\tau_k}}$ is not too big.  To control their size, we have to choose $\eta_k$ carefully, and the key bound is the $\| \widecheck{\eta}_k \|_{L^1} \le 1 + C / \log R$, which is proved in Lemma~\ref{etalem}. 

Finally, let us briefly discuss pruning wave packets.  Our argument involves many different scales and we have to prune wave packets at all of the scales. We can define $f_N$ to be our initial function $f$.  We decompose $f_N$ into wave packets by combining \eqref{eq: decomftauk} and  \eqref{eq: babywp}

$$ f_N = \sum_{\theta} \sum_T \psi_T f_{N, \theta}. $$

\noindent Then we remove the wave packets with amplitude bigger than $\lambda= 100Kr/a$.   The resulting function is called $f_{N-1}$:

$$ f_{N-1} = \sum_{\theta, T: \| \psi_T f_{N, \theta} \|_{\infty} \le \lambda } \psi_T f_{N, \theta}. $$

Next, we decompose $f_{N-1}$ into wave packets at the next scale:

$$ f_{N-1} = \sum_{\tau_{N-1}} \sum_{T_{\tau_{N-1}}} \psi_{T_{\tau_{N-1}}} f_{N-1, \tau_{N-1}}. $$

\noindent Here $\tau_{N-1}$ is a rectangle of dimensions $R_{N-1}^{-1/2} \times R_{N-1}^{-1}$, and $T_{\tau_{N-1}}$ is roughly a tube which is roughly a translate of $\tau_{N-1}^*$.  We remove the wave packets with amplitude bigger than $\lambda$ and call the resulting function $f_{N-2}$.  This iterative pruning is necessary to make our argument work, but it also makes it fairly complex.  In particular, since the pruning has $N$ steps, we have to be very careful with all the estimates related to the pruning process.  For example we have to define the smooth cutoff functions $\psi_T$ carefully.

\section{Proof of Theorem \ref{Main} -- the broad, well-spaced case \label{speccase}}

The argument outlined in the above intuition section leads to the $(\log R)^c$ upper bound in Theorem \ref{Main} for functions which satisfy two extra properties. The function $f$ being \emph{broad} allows us to bound an $L^4$ norm of $f$ by an $L^2$ norm of a square function $g_k$. The property that $f$ is \emph{well-spaced} allows us to replace Lemma \ref{babylow} with
\[ |g_{k,\ell}|\le \sum_{\tau_{k+1}}|f_{\tau_{k+1}}|^2*|\widecheck{\eta}_k| \]
(so we have no accumulated constant after iterating the inequality $\lesssim \log R$ times). Theorem \ref{Main} in the special case of broad, well-spaced functions $f$ is called Proposition \ref{mainp}, which we prove in this section. In \textsection\ref{remove}, we remove the assumptions on $f$.

\subsection{Statement of Proposition \ref{mainp}}

Let $f\in\mc{S}$ have Fourier support in $\mc{N}_{R^{-1}}(\mb{P}^1)$. $\theta$ is always an approximate $R^{-1/2}\times R^{-1}$ rectangle in a neighborhood of $\mb{P}^1$. The property that $f$ is broad means that $\|f\|_{L^6(Q_R)}$ is dominated by the $L^6$-norm of a bilinearized version of $f$. We state the results in terms of a parameter $\a>0$ which measures this bilinearized version of $f$. Precisely, let

\begin{definition} \label{defUa}
\[ U_\a:=\{x\in  \R^2: \max_{\substack{\tau,\tau'\\\text{nonadj.}}}|f_{\tau^{\,}}f_{\tau'}|^{1/2}(x)\sim \a\quad\text{and}\quad \big(\sum_{\tau}|f_{\tau}(x)|^6\big)^{1/6}\le (\log R)^9 \a\} \]
where the maximum is taken over nonadjacent $\sim (\log R)^{-6}\times (\log R)^{-12}$ rectangles $\tau^{\,}$ and $\tau'$.  By $\sim$ here, we mean within a factor of $2$. 
\end{definition}

\vspace*{.2in}

Our argument involves a sequence of scales $R_k$ defined as follows:

\begin{definition} For $k\in\N$, let $R_k=(\log R)^{12k}$. We analyze scale $R_1,\ldots,R_N$ where $R_N=R$. This means that $N=\frac{\log R}{12\log\log R}$. 
\end{definition}

\begin{definition}[Spacing property] Let $\Theta=\{\theta\}$ be a collection of $\sim R^{-1/2}\times R^{-1}$ rectangles contained in an $R^{-1}$-neighborhood of $\mb{P}^1$. The collection $\Theta$ has the spacing property at scale $R_k$ if there exists a collection of $\sim R_k^{-1/2}\times R_k^{-1}$-rectangles $\tau_k$ which cover $\cup_{\theta\in\Theta} \theta$ and such that 
\[\text{dist}(\tau_k,\tau_k')\ge\frac{1}{2} R_{k+1}^{-1/2} \]
whenever $\tau_k$ and $\tau_k'$ are distinct. If $\Theta$ has the spacing property at scales $R_1,\ldots,R_{N-1}$, then say $\Theta$ is well-spaced.  A function $f\in\mc{S}$ is well-spaced if $\hat{f}$ is supported in $\underset{\theta\in\Theta}{\cup}\theta$ for some well-spaced $\Theta$. 
\end{definition}

\vspace*{.2in}
\begin{prop} \label{mainp} There exist $c\in(0,\infty)$ such that for all well-spaced $f\in\mc{S}$,
 \[   \a^6|U_\a\cap Q_R|\le (\log R)^c (\sum_\theta\|f_\theta\|_{L^\infty(\mathbb{R}^2)}^2)^2 \sum_{\theta}\|f_\theta\|_{L^2(\mathbb{R}^2)}^2.\]
\end{prop}

\vspace*{.2in}

\begin{lemma}\label{lem: lc}For any $p\geq 1$, 
	$\|f_{\theta}\|_{L^{\infty}(\mathbb{R}^2)} \lesssim \|f_{\theta}\|_{L^{p}(\mathbb{R}^2)}.$
\end{lemma}
\begin{proof}
	Since the Fourier transform of $f_{\theta}$ is supported on $2\theta$, we can choose a smooth  cutoff function $\phi_{\theta}$ such that $\phi_{\theta}=1$ on $2\theta$ and $\phi_{\theta}=0$ outside of $3\theta$. 
	Now define $\psi_{\theta}(x) = \max_{y\in x+ B(0,2)} |\check{\phi}_{\theta}|$, then for any unit ball $B$
	\[
	\|f_{\theta}\|_{L^{\infty}(B)} \leq \min_{x\in B} |f_{\theta}|* \psi_{\theta}(x)  \leq \||f_{\theta}|* \psi_{\theta}\|_{L^p(B)} \lesssim \|f_{\theta}\|_{L^p(\mathbb{R}^2)}. 
	\]
\end{proof}
\begin{note} \label{normalization}
	For the remainder of \textsection\ref{speccase}, assume that we have replaced $f$ by a constant multiple $cf$ so that $\underset{\theta}{\max}\|f_\theta\|_{L^\infty(\mathbb{R}^2)}= 1$. Note that this means that $\a$ is replaced by $c\a$ and $r$ is replaced by $c^2r$. The purpose of this assumption is to simplify the error terms which are often written as negative powers of $R$. Note for example that by Lemma~\ref{lem: lc}
	\[R^{-50} \ll \frac{r^2}{\a^2}\max_\theta\|f_\theta\|_{L^\infty(\mathbb{R}^2)}^2\lesssim \frac{r^2}{\a^2}\sum_\theta\|f_\theta\|_{L^2(\mathbb{R}^2)}^2 , \]
	because $r \ge 1$ and $\a \le R^{1/2}$. 
\end{note}

\subsection{Auxiliary functions}

There are two places described in \textsection\ref{intuition} that involve auxiliary bump functions, Definition~\ref{def: bumpetak} and \eqref{eq: babypsik}, which we analyze carefully in this section. To formally carry out the pruning process from $f$ to $f_k$, we define $\p_{\tilde{T}_{\tau_k}}$, and to define the high/low decomposition of $g_k$, define $\widecheck{\eta}_k$. We control the $L^1$ norms of these functions in Lemma \ref{etalem} and Lemma \ref{P}. This is important for achieving the $(\log R)^c$ upper bound in Proposition \ref{mainp}. 
\vspace*{2mm}

\begin{notation} Let $\delta=\frac{1}{\log R}$.
\end{notation}

\begin{definition} \label{eta}
	 Let $\xi\in\R$ and $G_0(\xi)=e^{-(\log R)^{1/2}\xi^2}$. Define the Gaussian-like function 
\[ G(\xi)=G_0(\xi)\chi_{[-1,1]\setminus[-\delta,\delta]}(\xi)+G_0(\delta)\chi_{[-\delta,\delta]}(\xi)-G_0(1)\chi_{[-1,1]}(\xi). \]
Define $\eta:\R^2\to\R$ by 
\[ \eta(\xi_1,\xi_2)=G(0)^{-2}G(\xi_1)G(\xi_2).\]

For each $k$, define 
\begin{equation}\label{etadef} \eta_k(\xi)=\eta(4R_{k+2}^{1/2}\xi). 
\end{equation}
Note that $\eta_k(\xi)=1$ for all $|\xi_i|\le \frac{1}{4\log R}R_{k+2}^{-1/2}$ and $\eta_k$ is supported in $|\xi_i|\le \frac{1}{4}R_{k+2}^{-1/2}$. 
\end{definition}

\vspace*{2mm}
\begin{lemma}\label{etalem} Let $R>0$ be larger than a certain absolute constant. Then for each $k$,
\[ \|\widecheck{\eta}_k\|_1\le 1+\frac{c}{\log R}. \]
\end{lemma}

\begin{proof} It suffices to show the claim for $\eta$ since $\eta_k$ is equal to $\eta$ composed with an affine transformation. Recall the definition of $\eta$
\[ \eta(\xi_1,\xi_2)=G(0)^{-2}G(\xi_1)G(\xi_2). \]
Since $G(0)^{-1}=(e^{\delta^{3/2}}-e^{-\delta^{-1/2}})^{-1}\le (1+\frac{1}{\log R})$ and Fourier transforms of products factor, it suffices to show that
\[ \|\widecheck{G}\|_1\le 1+\frac{c}{\log R}. \]
Recall the definition of $G$: $G_0(\xi)=e^{-(\log R)^{1/2}\xi^2}$ and 
\[ G(\xi)=G_0(\xi)\chi_{[-1,1]\setminus[-\delta,\delta]}(\xi)+G_0(\delta)\chi_{[-\delta,\delta]}(\xi)-G_0(1)\chi_{[-1,1]}(\xi). \]
Also define the functions 
\[ G_1(\xi)=(G_0(\xi)-G_0(\delta))\chi_{[-\delta,\delta]}(\xi),\qquad G_2(\xi)=G_0(\xi)\chi_{[-1,1]^c}(\xi)+G_0(1)\chi_{[-1,1]}(\xi)   \]
and note that
\[ G(\xi)= G_0(\xi)-G_1(\xi)-G_2(\xi).   \]
Since $\widecheck{G_0}\ge 0$, $\|\widecheck{G_0}\|_1=\|G_0\|_\infty=1$, which means it suffices to show that
\[ \|\widecheck{G_1}\|_1\lesssim \frac{1}{\log R}\]
and
\[ \|\widecheck{G_2}\|_1\lesssim \frac{1}{\log R}. \]
Observe that $G_1$ and $G_2$ are continuous, $L^1$ functions (though not differentiable at a few points). $G_1$ is Riemann-integrable, so for each $x\not=0$, we can use integration by parts to compute $x\widecheck{G}_1(x)$ as the inverse Fourier transform of a function. For $G_2$, the same is true after a limiting argument to approximate $G_2$ by Riemann-integrable functions. 

First, we have
\begin{align*}
    \|\widecheck{G_1}\|_1&=\|(1+x^2)^{-1/2}(1+x^2)^{1/2}\widecheck{G_1}\|_1\\
    &\lesssim \|(1+x^2)^{1/2}\widecheck{G_1}\|_2\\
    &\lesssim \|G_1\|_2+\|x\widecheck{G_1}\|_2\\
    &\lesssim (\delta|1-G_0(\delta)|^2)^{1/2}+(\log R)^{1/2}\left(\int_{[-\delta,\delta]}|\xi e^{-(\log R)^{1/2}\xi^2}|^2d\xi\right)^{1/2}\\
    &\lesssim \delta^{1/2}(\log R)^{1/2}\delta^2+(\log R)^{1/2}\delta^{3/2}\sim \frac{1}{\log R}.
\end{align*}
Similarly, 
\begin{align*}
    \|\widecheck{G_2}\|_1 &\lesssim \|G_2\|_2+\|x\widecheck{G_2}\|_2\\
    &\lesssim G_0(1)+(\log R)^{1/2}\left(\int_{[-1,1]^c}|\xi e^{-(\log R)^{1/2}\xi^2}|^2d\xi\right)^{1/2}\\
    &\lesssim e^{-(\log R)^{1/2}}+(\log R)^{1/8}\big(\int_{[-\delta^{-1/4},\delta^{-1/4}]^c}e^{-\xi^2/2}d\xi\big)^{1/2}\ll \frac{1}{\log R}. 
\end{align*}

\end{proof}

\vspace*{2mm}

\begin{definition}\label{rho} Let $\rho:\R^2\to\R$ be equal $\rho(\xi)=\eta(\delta \xi)$ where $\eta$ is defined in Definition \ref{eta} and $\delta = \frac{1}{\log R}$. For each scale $R_k$ and each $\sim R_k^{-1/2}\times R_k^{-1}$-rectangle $\tau_k$, let 
\[ \rho_{\tau_k}=\rho\circ\ell_{\tau_k}\]
where $\ell_{\tau_k}$ is an affine transformation mapping the smallest ellipse containing $2\tau_k$ to $\B$. 
\end{definition}


\begin{definition}\label{def: dual convex}
	If $\tau$ is a symmetric convex set with center $C(\tau)$, then  the dual of $\tau$ is defined as
	\[
	\tau^* =\{x: |x\cdot (y-C(\tau))|\leq 1, \, \, \forall y\in \tau\}.
	\]
\end{definition}

\begin{definition} \label{aux} For each $R_1^{-1/2}\times R_1^{-1}$-rectangle $\tau_1$, let 
\[ \p_{\tilde{T}_{\tau_1}}(x)=\sup_{y\in x+(\log R)^{11} \tau_1^*}|\widecheck{\rho}_{\tau_1}(y)|.  \]
For each $2\le k\le N$ and each $R_k^{-1/2}\times R_k^{-1}$-rectangle $\tau_k$, define $\p_{\tilde{T}_{\tau_k}}$ inductively by 
\[ \p_{\tilde{T}_{\tau_k}}(x)=\max(\sup_{y\in x+(\log R)^{11}\tau_k^* }|\widecheck{\rho}_{\tau_k}(y)|,\p_{\tilde{T}_{\tau_{k-1}}}*|\widecheck{\eta}_{k-1}|(x)),\]
where $\tau_{k-1}$ is the $R_{k-1}^{-1/2}\times R_{k-1}^{-1}$-rectangle containing $\tau_k$.
\end{definition}

\vspace*{2mm}

\begin{lemma}\label{P} For each $k$ and $\tau_k$,
\[ \|\p_{\tilde{T}_{\tau_k}}\|_1\lesssim (\log R)^{\tilde{c}}. \]

The implicit constant is uniform in $k$ and $\tau_k$.
\end{lemma}
\begin{proof} If $k=1$, then $\|\p_{\tilde{T}_{\tau_1}}\|_1\lesssim (\log R)^{\tilde{c}}$ is clear from the definition. Recall the definition of $\p_{\tilde{T}_{\tau_k}}$ for $k\ge 2$ to be
\[ \p_{\tilde{T}_{\tau_k}}(x)=\max(\sup_{y\in x+(\log R)^{11} \tau_k^* }|\widecheck{\rho}_{\tau_k}(y)|,\p_{\tilde{T}_{\tau_{k-1}}}*|\widecheck{\eta}_{k-1}|(x)),\]
where $\tau_{k-1}$ is the $R_{k-1}^{-1/2}\times R_{k-1}^{-1}$-rectangle containing $\tau_k$. Let $A_k$ be the set on which 
\[ \sup_{y\in x+(\log R)^{11}\tau_k^*}|\widecheck{\rho}_{\tau_k}(y)|\ge \p_{\tilde{T}_{\tau_{k-1}}}*|\widecheck{\eta}_{k-1}|(x)\]
and let $B_k$ be the complement of $A_k$. Note that
\[ \|\p_{\tilde{T}_{\tau_k}}\|_1=\|\sup_{y\in \cdot +(\log R)^{11} \tau_k^* }|\widecheck{\rho}_{\tau_k}(y)|\|_{L^1(A_k)}+ \|\p_{\tilde{T}_{\tau_{k-1}}}*|\widecheck{\eta}_{k-1}|\|_{L^1(B_k)}\]
If 
\[ \|\sup_{y\in \cdot +(\log R)^{11}\tau_k^*}|\widecheck{\rho}_{\tau_k}(y)|\|_{L^1(A_k)}\ge \delta  \|\p_{\tilde{T}_{\tau_{k-1}}}*|\widecheck{\eta}_{k-1}|\|_{L^1(B_k)}, \]
then 
\begin{align*}
 \|\p_{\tilde{T}_{\tau_k}}\|_1 & \le 2 \delta^{-1}\|\sup_{y\in \cdot +(\log R)^{11}\tau_k^*}|\widecheck{\rho}_{\tau_k}(y)| \|_{L^1(\RR^2)} \\
& = 2 \delta^{-1} \| \sup_{y \in \cdot + B(0,  (\log R )^{11} )} | \widecheck{\rho}(y)| \|_{L^1} \lesssim (\log R)^c 
\end{align*}
\noindent since $|\widecheck{\rho}|$ is bounded by $(\log R)^c$ and is rapidly decaying outside a ball of radius $(\log R)^c$.  

If not, then
\[ \|\p_{\tilde{T}_{\tau_k}}\|_1\le (1+\delta)\|\p_{\tilde{T}_{\tau_{k-1}}}*|\widecheck{\eta}_{k-1}|\|_{L^1(B_k)}. \]
By Young's convolution inequality and Lemma \ref{etalem}, 
\[ \|\p_{\tilde{T}_{\tau_k}}\|_1\le (1+\delta)^2\|\p_{\tilde{T}_{\tau_{k-1}}}\|_{1}  .\]

If we iterate this argument for $\leq N \lesssim \frac{\log R}{\log \log R}$ times, then we also obtain the desired conclusion. 

\end{proof}

\begin{definition}[\label{Gaussian}Gaussian partition of unity] First define $\s_0$ by
\[ \s_0(x)=c\int_{Q}g(x-y)dy\]
where $Q$ is the unit cube $Q=[-1/2,1/2]^2$, $g(x)=e^{-|x|^2}$, and $c=(\int g)^{-1}$. Note that for every $x$
\[ \sum_{n\in\Z^2}\s_0(x-n)=\sum_{n\in\Z^2}\int_{n+Q} g(x-y)dy=1. \]

Let $T$ be any rectangle in $\R^2$. Let $A:\R^2\to\R^2$ be an affine transformation mapping $T$ to $Q$. Define the Gaussian bump function adapted to $T$ $\s_T$ by
\[  \s_T(x)=c|T|^{-1}\int_Tg(A(x-y))dy. \]
(There are several different affine transformations $A$ taking $T$ to $Q$, but they all give the same function $\s_T$ because the Gaussian $g$ and the area form are invariant under the affine automorphisms of $Q$.)

If $\T$ is a set of congruent rectangles $T$ tiling the plane, then 
\[ 
\sum_{T \in \T} \s_T(x) = 1,
 \]
and so the Gaussians $\{ \s_T \}_{T \in \T}$ form a partition of unity.  

The Fourier transform of $\s_T$ is
\[ \widehat{\s_T}(\xi)=c|T|^{-1}\int_T |A|^{-1}\widehat{g}((A^{-1})^t\xi)e^{-2\pi i \xi\cdot A^{-1} y}dy   \]
and satisfies
\[ |\widehat{\s_T}(\xi)|\le |T|e^{-|(A^{-1})^t\xi|^2}. \]
If $x\not\in100\sqrt{\log R}T$, then 
\[ |\s_T(x)|\le ce^{-\underset{y\in T}{\inf}|x-y|^2} \le R^{-1000}. \]
If $\xi\not\in 100\sqrt{\log R}T^*$, then
\[ |\widehat{\s_T}(\xi)|\le c|T|R^{-10^4}. \]

\end{definition}

	In the rest of the paper, a function being \emph{essentially supported} in $S$ means  that $|\cdot|\le R^{-1000}$ off of $S$ and the function rapidly decays away from $S$.  
A weight function $w_S$ is localized to $S$ means that $w_S\sim 1$ on $S$ and $|w_S|\leq R^{-1000}$ off of $\log R \, S$. 

\subsection{Pruning wave packets \label{lemmas} } 

We define pruned versions of the function $f$ and the intermediate square functions $\underset{\tau_k}{\sum}|f_{\tau_k}|^2$. The pruning process depends on the parameter $\a>0$ which measures the bilinearized version of $f$, and a new parameter $r>0$ (defined below) related to the final square function $\underset{\theta}{\sum}|f_\theta|^2$.

\begin{definition} Define $g_N$ as
\[ g_N=\sum_\theta|f_\theta|^2*\p_{\tilde{T}_{\theta}}. \]
where $\p_{\tilde{T}_{\theta}}$ is defined in Definition \ref{aux} for $\theta=\tau_N$. \end{definition}

\begin{definition} Define $r>0$ to be
\begin{equation}\label{eq: r}
r = \|g_N\|_{L^\infty(\mathbb{R}^2)}. \end{equation}
Note that $r$ satisfies 
\begin{equation} 
r\lesssim (\log R)^{\tilde{c}}\sum_\theta\|f_\theta\|_{L^\infty(\mathbb{R}^2)}^2  \end{equation}
for the constant $\tilde{c}$ in Lemma~\ref{P}.
\end{definition}

\begin{notation} \label{deflambda} The parameter $\lambda$ measures the ratio between $r$ and $\a$:
\begin{equation} \lambda=(\log R)^m \frac{r}{\a}\end{equation}
where the exponent $m$ is sufficiently large as required by the proof of Lemma \ref{hkbound} and Lemma \ref{ftofk} and Proposition \ref{mainp}. 
\end{notation}

\vspace*{.2in} 

\begin{definition} \label{defTtauk} For each rectangle $\tau_k$, we write $T_{\tau_k}$ to denote a translate of $(\log R)^9 \tau_k^*$.  
We let $\T_{\tau_k}$ be a tiling of the plane by rectangles $T_{\tau_k}$.  
In Definition \ref{Gaussian}, we defined a Gaussian partition of unity associated to such a tiling:

$$ \sum_{T_{\tau_k} \in \T_{\tau_k}} \s_{T_{\tau_k}}(x) = 1. $$

\end{definition}

\begin{definition}[Defining $f_k$ with respect to $\lambda$] \label{fk} Define $f_N$ by
\[ f_N=\sum_\theta\sum_{\substack{T_\theta \in\T_\theta\\\|\s_{T_\theta}f_\theta\|_\infty\le\lambda}}\s_{T_\theta}f_\theta. \]

Let $\T_{\theta,\lambda}=\{T_\theta \in\T_\theta:\|\s_{T_\theta}f_\theta\|_\infty\le\lambda\}$. Define 
\[ f_{N,\theta}=\sum_{T_\theta \in\T_{\theta,\lambda}}\s_{T_\theta}f_\theta\]
and note that $\widehat{f_{N,\theta}}$ is essentially supported  in $(1+(\log R)^{-8}) \theta$. Also define
\[ f_{N,\tau_{N-1}}=\sum_{\theta\subset \tau_{N-1}}f_{N,\theta} .\]

Now we define $f_k,f_{k,\tau_k},$ and $f_{k,\tau_{k-1}}$, starting with $k=N-1$ and going down to $k=1$.

\[ f_k =\sum_{\tau_k}\sum_{T_{\tau_k}\in\T_{\tau_k,\lambda}}\s_{T_{\tau_k}}f_{k+1,\tau_k} \]

\noindent where $\T_{\tau_k,\lambda}=\{T_{\tau_k}\in\T_{\tau_k}:\|\s_{T_{\tau_k}}f_{k+1,\tau_k}\|_\infty\le\lambda\}$. 

Similarly we define 

\begin{equation} \label{eq: fktauk}
f_{k,\tau_k} =\sum_{T_{\tau_k}\in\T_{\tau_k,\lambda}}\s_{T_{\tau_k}}f_{k+1,\tau_k}
\end{equation}
\begin{equation}\label{eq: fktauk1}
f_{k,\tau_{k-1}} = \sum_{\tau_k\subset\tau_{k-1}} f_{k,\tau_{k}}.
\end{equation} 

\end{definition}

\begin{lemma}[Properties of $f_k$]\label{propfk} 
\begin{enumerate} 
\item $f_k=\underset{\tau_k}{\sum}f_{k,\tau_k}.$
\item $ | f_{k, \tau_{k}} (x) | \le |f_{k+1, \tau_{k}}(x)|. $
\item $ \| f_{k, \tau_k} \|_{L^\infty} \le C (\log R)^2 \lambda + R^{-1000}$.
\item\label{item3} $ \text{ess supp} \widehat{f_{k,\tau_k}}\subset (1+(\log R)^{-8})\cdot\tau_k . $
\item \label{item4} $  \text{ess supp} \widehat{f_{k,\tau_{k-1}}}\subset (1+(\log R)^{-10})\tau_{k-1}. $

\end{enumerate}

\end{lemma}

\begin{proof}
The first property follows straight from the definition.

The second property follows because $\sum_{T_{\tau_k} \in \T_{\tau_k}}  \s_{T_{\tau_k}}$ is a partition of unity, and

$$ f_{k,\tau_k}=\sum_{T_{\tau_k}\in\T_{\tau_k,\lambda} \subset \T_{\tau_k}} \s_{T_{\tau_k}}f_{k+1,\tau_k}. $$

Now consider the $L^\infty$ bound in the third property.  We write

$$ f_{k, \tau_k}(x) = \sum_{T_{\tau_k} \in \T_{\tau_k, \lambda}, x \in (\log R) T_{\tau_k}} \s_{T_{\tau_k}} f_{k+1, \tau_k} + \sum_{T_{\tau_k} \in \T_{\tau_k, \lambda}, x \notin (\log R) T_{\tau_k}} \s_{T_{\tau_k}} f_{k+1, \tau_k}. $$

\noindent The first sum has at most $C (\log R)^2$ terms, and each term has norm bounded by $\lambda$ by the definition of $\T_{\tau_k, \lambda}$.  By the normalization in Note \ref{normalization}, it follows easily that 
\begin{equation}\label{eq: trivialbd}
|f_{k+1, \tau_k}(x)| \le R^{100}.
\end{equation}
 But if $x \notin (\log R) T_{\tau_k}$, then $\s_{T_{\tau_k}}(x) \le e^{- (\log R)^2} \le R^{-2000}$.  Moreover, as $T_{\tau_k}$ gets further away from $x$, $\s_{T_{\tau_k}}(x)$ is rapidly decaying.  Therefore, the second sum has norm at most $R^{-1000}$.

The fourth and fifth properties depend on the essential Fourier support of $\s_{T_{\tau_k}}$ (and on similar trivial bounds as (\ref{eq: trivialbd})).  Recall from Definition \ref{defTtauk} that $T_{\tau_k}$ is a translate of $(\log R)^9 \tau_k^*$.  Because of this factor $(\log R)^9$, the essential Fourier support of $\s_{T_{\tau_k}}$ is contained in \newline $100\sqrt{\log R}(\log R)^{-9} \tau_k$ (see Definition~\ref{Gaussian}).  

Initiate a 2-step induction with base case $k=N$: $f_{N,\theta}$ has essential Fourier support in $(1+(\log R)^{-8})\theta$ because of the above definition. Then 
\[ f_{N,\tau_{N-1}}=\sum_{\theta\subset\tau_{N-1}}f_{N,\theta}\]
has essential Fourier support in $\underset{\theta\subset\tau_{N-1}}{\cup}(1+(\log R)^{-8})\theta$, which is contained in $(1+(\log R)^{-10})\tau_{N-1}$. Since each $\s_{T_{\tau_{N-1}}}$ has essential Fourier support in $100\sqrt{\log R}(\log R)^{-9}\tau_{N-1}$, 
\[ f_{N-1,\tau_{N-1}}=\sum_{T_{\tau_{N-1}}\in\T_{\tau_{N-1},\lambda}}\s_{\tau_{N-1}}f_{N,\tau_{N-1}}\]
has essential Fourier support in 
\[ (100\sqrt{\log R}(\log R)^{-9}+1+(\log R)^{-10})\tau_{N-1}\subset(1+(\log R)^{-8})\tau_{N-1}. \]
Iterating this reasoning until $k=1$ gives \eqref{item3} and \eqref{item4}.

\end{proof}

\begin{definition}[Definition of $g_k$] \label{def: gk}
	For $k=1,\ldots,N-1$, define
\[ g_k:=\sum_{\tau_k}|f_{k+1,\tau_k}|^2*\p_{\tilde{T}_{\tau_k}} \]
where $\p_{\tilde{T}_{\tau_k}}$ is specified in Definition \ref{aux}. \end{definition}

The following lemma shows that the difference between the $k$th and $(k+1)$st versions of $f_{\tau}$ is controlled by $\lambda^{-1}g_k$. We eventually apply this lemma for $x\in\Omega_k$, defined in Definition \ref{omega}, where we know that $g_k\sim r$. We will see that on this set, the differences between the different versions of $f_{\tau}$ are negligible.
\vspace*{.2in} 

\begin{lemma}\label{hkbound} Suppose $\tau$ is a $\rho^{-1/2} \times \rho^{-1}$ rectangle in a $\rho^{-1}$-neighborhood of $\mb{P}^1$, at any scale $1 \le \rho \le R$.  For $k=1,\ldots,N-1$, if $R_k \ge \rho$, 
\[ |\sum_{\tau_k\subset\tau}f_{k+1,\tau_k}(x)- \sum_{\tau_k \subset \tau} f_{k,\tau_k}(x)|\lesssim (\log R)^{\tilde{c}}\lambda^{-1}g_k(x)+R^{-1000}. \]
\end{lemma}
\begin{proof} By \eqref{eq: fktauk},  
\[ \sum_{\tau_k\subset\tau}f_{k+1,\tau_k}(x)-\sum_{\tau_k\subset\tau}f_{k,\tau_k}(x)=\sum_{\tau_k\subset\tau}\sum_{\substack{T\in \T_{\tau_k}\\ \|\s_Tf_{k+1,\tau_k}\|_\infty>\lambda}}\s_T(x)f_{k+1,\tau_k}(x). \]
Recall that $\s_T$ has Gaussian decay off of $T$. It follows from \eqref{eq: trivialbd} that \begin{align*}
 \big|\sum_{\tau_k\subset\tau}\sum_{\substack{T\in\T_{\tau_k}\\\|\s_Tf_{k+1,\tau_k}\|_\infty>\lambda\\x\not\in\log R\cdot T}}\s_T(x)f_{k+1,\tau_k}(x)\big|\le R^{-1000}. 
\end{align*}
Then we have
\begin{align*} 
\big|\sum_{\tau_k\subset\tau}\sum_{\substack{T\in\T_{\tau_k}\\\|\s_Tf_{k+1,\tau_k}\|_\infty>\lambda\\x\in\log R\cdot T}}\s_T(x)f_{k+1,\tau_k}(x)\big|&\le \lambda^{-1}\sum_{\tau_k\subset\tau}\sum_{\substack{T\in\T_{\tau_k}\\\|\s_Tf_{k+1,\tau_k}\|_\infty>\lambda\\x\in\log R\cdot T}}\|\s_Tf_{k+1,\tau_k}\|_\infty^2 \\
\end{align*}

The number of terms in the inner sum is $\lesssim (\log R)^2$.  Let $T_{\tau_k}(x)$ be the unique rectangle in the tiling $\T_{\tau_k}$ that includes $x$.  If $x \in (\log R) T$, then $\s_T$ is essentially supported in $10 (\log R) T_{\tau_k}(x)$, and so we have

\[ 
\sum_{\tau_k\subset\tau}\sum_{\substack{T\in\T_{\tau_k}\\\|\s_Tf_{k+1,\tau_k}\|_\infty>\lambda\\x\in\log R\cdot T}}\|\s_Tf_{k+1,\tau_k}\|_\infty^2 \lesssim (\log R)^2 \lambda^{-1} \sum_{\tau_k \subset \tau} \| f_{k+1, \tau_k} \|_{L^\infty( 10 (\log R) T_{\tau_k}(x))}^2 + R^{-1000}. \]

Recall by Lemma \ref{propfk} that the Fourier transform of $f_{k+1,\tau_k}$ is essentially supported in $(1+(\log R)^{-10})\tau_k$, so for $\rho_{\tau_k}$ from Definition \ref{rho},
\[ |f_{k+1,\tau_k}(y)|\le|f_{k+1,\tau_k}*\widecheck{\rho}_{\tau_k}(y)|+R^{-1000}.\]

Now we defined $\p_{\tilde T_{\tau_k}}(z)$ to be at least $ \underset{z + (\log R)^{11}\tau_k^*}{\sup} | \widecheck{\rho}_{\tau_k}|$, so

\begin{equation}\label{eq: lc3201}
\| f_{k+1, \tau_k} \|_{L^\infty( 10 (\log R) T_{\tau_k}(x))} \le |f_{k+1, \tau_k}| * \p_{\tilde T_{\tau_k}} (x). 
\end{equation}
Therefore,

\[ 
\big|\sum_{\tau_k\subset\tau}\sum_{\substack{T\in\T_{\tau_k}\\\|\s_Tf_{k+1,\tau_k}\|_\infty>\lambda\\x\in\log R\cdot T}}\s_T(x)f_{k+1,\tau_k}(x)\big| \lesssim (\log R)^2 \lambda^{-1} \sum_{\tau_k \subset \tau} \left( | f_{k+1, \tau_k}| * \p_{\tilde T_{\tau_k}} (x) \right)^2 +  R^{-1000}. \]
Applying Cauchy-Schwarz to the integral in the convolution, we get
\[ \le (\log R)^2 \lambda^{-1} \sum_{\tau_k \subset \tau}  | f_{k+1, \tau_k}|^2 * \p_{\tilde T_{\tau_k}} (x) \| \p_{\tilde T_{\tau_k}} \|_{L^1} +  R^{-1000}. \]
Finally, note that $\| \p_{\tilde T_{\tau_k}} \|_{L^1} \le (\log R)^c$ by Lemma \ref{P} and that
\begin{equation}\label{eq: lc3202}
\sum_{\tau_k \subset \tau}  | f_{k+1, \tau_k}|^2 * \p_{\tilde T_{\tau_k}} (x) \le g_k(x)
\end{equation}
 by Definition~\ref{def: gk}.
\end{proof}

\subsection{High/low lemmas for $g_k$}

\begin{definition}[Definition of $g_k^{\ell}$ and $g_k^h$]\label{def: gkhl}
	 For $k=1,\ldots,N-1$ and $\eta_k$ from \eqref{etadef}, define
\[ g_k^{\ell}:=g_k*\widecheck{\eta}_k\quad\text{and}\quad g_k^h:=g_k-g_k^{\ell}.  \]
\end{definition}

\begin{lemma}[Low lemma]\label{low}
\[ |g_k^{\ell}|\le g_{k+1} +R^{-1000}\]
\end{lemma} 

\begin{proof} Write
\[ |g_k^{\ell}(x)|=|\sum_{\tau_k}|f_{k+1,\tau_k}|^2*\p_{\tilde{T}_{\tau_k}}*\widecheck{\eta}_k(x)| \]
where $\eta_k$ is defined in Definition~\ref{eta}. Analyze for each $\tau_k$:
\begin{align*}
    |f_{k+1,\tau_k}|^2*\widecheck{\eta}_k(x)&= \int\widehat{f}_{k+1,\tau_k}*\widehat{\overline{f}}_{k+1,\tau_k}(\xi) e^{2\pi ix\cdot\xi}\eta_k(\xi)d\xi  \\
    &= \sum_{\tau_{k+1},\tau_{k+1}'\subset\tau_k}\int\widehat{f}_{k+1,\tau_{k+1}}*\widehat{\overline{f}}_{k+1,\tau_{k+1}'}(\xi) e^{2\pi ix\cdot\xi}\eta_k(\xi)d\xi.   
\end{align*}
The function $\widehat{f}_{k+1,\tau_{k+1}}*\widehat{\overline{f}}_{k+1,\tau_{k+1}'}$ is essentially supported in $(1+(\log R)^{-8})(\tau_{k+1}-\tau_{k+1}')$. By the well-spaced property, we know that this set does not intersect the ball of radius $\frac{1}{2}R_{k+2}^{-1/2}$ (the support of $\eta_k$) unless $\tau_k=\tau_k'$.  Therefore,
up to errors of size $R^{-1000}$, we have

\[    |f_{k+1,\tau_k}|^2*\widecheck{\eta}_k(x) =  \sum_{\tau_{k+1}\subset\tau_k}\int\widehat{f}_{k+1,\tau_{k+1}}*\widehat{\overline{f}}_{k+1,\tau_{k+1}}(\xi) e^{2\pi ix\cdot\xi}\eta_k(\xi)d\xi = \sum_{\tau_{k+1} \subset \tau_k} |f_{k+1, \tau_{k+1}}|^2 * \widecheck{\eta}_k(x). \]

By Lemma \ref{propfk}, $|f_{k+1,\tau_{k+1}}|\le |f_{k+2,\tau_{k+1}}|$, and so

\[   | |f_{k+1,\tau_k}|^2*\widecheck{\eta}_k(x) | \le  \sum_{\tau_{k+1} \subset \tau_k} |f_{k+2, \tau_{k+1}}|^2 * |\widecheck{\eta}_k|(x). \]

Plug this back in to the definition of $g_k^{\ell}$ to get

\[ |g_k^{\ell}(x)| \le \sum_{\tau_k} ||f_{k+1,\tau_k}|^2*\widecheck{\eta}_k| *\p_{\tilde T_{\tau_k}} \le \]

\[ \le \sum_{\tau_{k+1}} |f_{k+2, \tau_{k+1}}|^2 * |\widecheck{\eta}_k| * \p_{\tilde T_{\tau_k}}. \]

Now $\p_{\tilde T_{\tau_{k+1}}}$ was defined  in Definition~\ref{aux} so that when $\tau_{k+1} \subset \tau_k$,

\[  |\widecheck{\eta}_k| * \p_{\tilde T_{\tau_k}} \le \p_{\tilde T_{\tau_{k+1}}}. \]

Plugging that in again, we get

\[ | g_k^\ell | \le \sum_{\tau_{k+1}} |f_{k+2, \tau_{k+1}}|^2 *  \p_{\tilde T_{\tau_{k+1}}} = g_{k+1}. \]

\end{proof}


\begin{lemma}[High lemma] \label{hilem}
\[\int|g_k^h|^2 \lesssim (\log R)^{\tilde{c}}\int\sum_{\tau_k}|f_{k+1,\tau_k}|^4+R^{-1000} . \]
\end{lemma}

\begin{proof} By Definition~\ref{def: gk} and \ref{def: gkhl}, 
\begin{align*} 
\int |g_k^h|^2 &= \sum_{\tau_k}\sum_{\tau_k'}\int(|f_{k+1,\tau_k}|^2)^{\widehat{\,\,}}\,\,\widehat{\p}_{\tilde{T}_{\tau_k}}(1-\eta_k) (|f_{k+1,\tau_k'}|^2)^{\widehat{\,\,}}\,\,\widehat{\p}_{\tilde{T}_{\tau_k'}}(1-\eta_k).
\end{align*} 
The Fourier transform of $|f_{k+1,\tau_k}|^2$ is essentially supported in $2(\tau_k-\tau_k)$. Recall also from Definition~\ref{eta}  that $(1-\eta_k)$ is supported where $|\xi|\ge \frac{1}{4\log R}R_{k+2}^{-1/2}$.    The set $2(\tau_k-\tau_k)\setminus B(\frac{1}{4\log R}R_{k+2}^{-1/2})$ overlaps at most $\sim \log R R_{k+2}^{1/2}/R_k^{1/2}=(\log R)^{13}$ many of the sets $2(\tau_k'-\tau_k')\setminus B(\frac{1}{4\log R}R_{k+2}^{-1/2})$. Thus applying Cauchy-Schwarz to the integral in the convolution, we get
\begin{align*}  
\int \big|g_k^h\big|^2 &\lesssim (\log R)^{\tilde{c}'}\int\sum_{\tau_k}||f_{k+1,\tau_k}|^2*\p_{\tilde{T}_{\tau_k}}|^2+R^{-1000}\\
&\lesssim (\log R)^{\tilde{c}'}\int\sum_{\tau_k}  \| \p_{\tilde{T}_{\tau_k}}\|_1^2 ||f_{k+1,\tau_k}|^4  +R^{-1000}\\
&\lesssim (\log R)^{\tilde{c}''}\int\sum_{\tau_k}|f_{k+1,\tau_k}|^4+R^{-1000} .
\end{align*} 
where we used that $\|\p_{\tilde{T}_{\tau_k}}\|_1\lesssim (\log R)^{\tilde{c}}$ (by Lemma \ref{P}).
\end{proof}

\subsection{The sets $\Omega_k$}
 
In this subsection, we will decompose the starting set $Q_R$ into $(Q_R\cap L)\cup(Q_R\cap\Omega_1)\cup\cdots\cup(Q_R\cap\Omega_{N-1})$. On the set $\Omega_k$, the bilinearized version of $f$ is basically the same as for $f_k$ (see Lemma \ref{ftofk}) and $g_k$ is high-dominated (see Lemma \ref{high}).

\vspace*{.2in}

\begin{definition}[Definition of $\Omega_k$]\label{omega} Recall the parameter $r>0$ defined in Definition \eqref{eq: r}. Let $\Omega_N=\cup Q_N$ be a collection of pairwise disjoint $R^{1/2}$-cubes which have nonempty intersection with $Q_R$. Let $\Omega_{N-1}$ be a collection of disjoint $R_{N-1}^{1/2}$-cubes $Q_{N-1}$ in $\Omega_N$ satisfying
\[(1+\delta)r+R^{-500} < \|g_{N-1}\|_{L^\infty((\log R)^9\,Q_{N-1})}.  \]

We define $\Omega_{k}$ for $k = N-2$, then $k=N-3$, down to $k=1$.  To define $\Omega_k$, partition $Q_R\setminus(\Omega_{N-1}\sqcup\cdots\sqcup\Omega_{k+1})$ into $R_k^{1/2}$-cubes $Q_{k}$. Define $\Omega_k$ to be the collection of $Q_{k}$ in the partition which satisfy 
\[  (1+\delta)^{N-k}r+(N-k)R^{-500} <\|g_k\|_{L^\infty((\log R)^9\,Q_k)}.  \]
Also define 
\[ L:=Q_R \setminus(\Omega_1\sqcup \Omega_2\sqcup\cdots\sqcup \Omega_{N-1}).\]

Recall that $\delta=\frac{1}{\log R}$. The absolute constant $C$ will be determined by the proof of Lemma \ref{high}. 
\end{definition}

\vspace*{.2in}

\begin{lemma}\label{ftofk}  Suppose $\tau$ is a $\rho^{-1/2} \times \rho^{-1}$ rectangle in a $\rho^{-1}$-neighborhood of $\mb{P}^1$, at any scale $1 \le \rho \le R$.  For $k=1,\ldots,N-1$, if $R_k \ge \rho$, 
\[ |f_{\tau}(x)-\sum_{\tau_k\subset\tau}f_{k+1,\tau_k}(x)|\le(\log R)^{-10}\a +R^{-500} \qquad \qquad \forall x\in U_\a\cap Q_R\cap \Omega_k. \]
\end{lemma}
\begin{proof} First note that for each $l\in\{1,\ldots,N-1\}$, by \eqref{eq: fktauk1}, 
\[ \sum_{\tau_{l}\subset\tau}f_{l,\tau_l}=\sum_{\tau_{l-1}\subset\tau}\sum_{\tau_l\subset\tau_{l-1}}f_{l,\tau_l} =\sum_{\tau_{l-1}\subset\tau}f_{l,\tau_{l-1}}. \]
Then we may decompose the difference as
\begin{align*}  
|f_{\tau}-\sum_{\tau_k\subset\tau}f_{k+1,\tau_k}|&=|f_{\tau}-\sum_{\theta\subset\tau}f_{N,\theta}|+|\sum_{\tau_{N-1}\subset\tau}f_{N,\tau_{N-1}}-\sum_{\tau_{N-1}\subset\tau}f_{N-1,\tau_{N-1}}|\\
&+\cdots+ |\sum_{\tau_{k+1}\subset\tau}f_{k+2,\tau_{k+1}}-\sum_{\tau_{k+1}\subset\tau }f_{k+1,\tau_{k+1}}|.
\end{align*}
By Lemma \ref{hkbound}, this is bounded by 
\begin{align*}
(\log R)^{\tilde{c}}\lambda^{-1}(g_N(x) +g_{N-1}(x)+\cdots+g_{k+1}(x))   +(N-k)R^{-1000}.
\end{align*} 
Finally, use the definition of $\Omega_k$ for $k\le N-1$ to say
\begin{equation} \label{ftofkeq} |f_{\tau}(x)-\sum_{\tau_k\subset\tau}f_{k+1,\tau_k}|\lesssim (\log R)^{\tilde{c}} (N-k)(1+\delta)^{N-k}\lambda^{-1}r+(N-k)^2R^{-500}  \end{equation}
Recall that $N \le \log R$, $(1+\delta)^N\sim 1$, and recall that $\lambda$ was defined in Notation \ref{deflambda} by
\[ \lambda = (\log R)^m r \alpha^{-1}. \]
By choosing $m$ sufficiently large, we can guarantee that the main term on the right-hand side of  inequality~\eqref{ftofkeq} is bounded by $(\log R)^{-10} \alpha$. 
\end{proof}

\begin{lemma}[$g_k$ is high-dominated on $\Omega_k$]\label{high} Let $k=1,\ldots,N-1$. For each $R_k^{1/2}$-cube $Q_k\subset \Omega_k$,  
\[ \|g_k\|_{L^\infty((\log R)^9\, Q_k)}\le 2(\log R)\|g_k^h\|_{L^\infty((\log R)^9\, Q_k)} .\]
\end{lemma}

\begin{proof} First consider the $k=N-1$ case. By definition of $\Omega_{N-1}$, for each $Q_{N-1}\subset\Omega_{N-1}$, 
\begin{equation}\label{A} (1+\delta)r+R^{-500} < \|g_{N-1}\|_{L^\infty((\log R)^9\,Q_{N-1})}. \end{equation}
Note that
\[ \|g_{N-1}\|_{L^\infty((\log R)^9\,Q_{N-1})}\le \|g_{N-1}^{\ell}\|_{L^\infty((\log R)^9\,Q_{N-1})}+\|g_{N-1}^h\|_{L^\infty((\log R)^9\,Q_{N-1})}\]
and suppose that 
\[ \|g_{N-1}^h\|_{L^\infty((\log R)^9\,Q_{N-1})}\le \delta\|g_{N-1}^{\ell}\|_{L^\infty((\log R)^9\,Q_{N-1})}. \]
Then by Lemma \ref{low} and \eqref{eq: r}, 
\begin{align*}
\|g_{N-1}\|_{L^\infty((\log R)^9\,Q_{N-1})}&\le (1+\delta)\|g_{N-1}^{\ell}\|_{L^\infty((\log R)^9\,Q_{N-1})}\\
&\le (1+\delta)\|g_N\|_{L^\infty((\log R)^9 \,Q_{N-1})}+(1+\delta)R^{-1000} \\
&\le (1+\delta)r+R^{-500}.
\end{align*}
Note that this contradicts (\ref{A}), so we have the desired conclusion. 

Now suppose that $k\in\{1,\ldots,N-2\}$. By definition of $\Omega_k$, 
\begin{equation}\label{B} (1+\delta)^{N-k}r+(N-k)R^{-500}< \|g_k\|_{L^\infty((\log R)^9\,Q_k)}. \end{equation}
Note that
\[ \|g_k\|_{L^\infty((\log R)^9\,Q_k)}\le \|g_k^{\ell}\|_{L^\infty((\log R)^9\,Q_k)}+\|g_k^h\|_{L^\infty((\log R)^9\,Q_k)}\]
and suppose that 
\begin{equation}\label{B'} \|g_k^h\|_{L^\infty((\log R)^9\,Q_k)}\le \delta\|g_k^{\ell}\|_{L^\infty((\log R)^9\,Q_k)}. \end{equation}
Then by Lemma \ref{low}, 
\begin{equation}\label{assump1}
    \|g_k\|_{L^\infty((\log R)^9\,Q_k)}\le (1+\delta)\|g_k^{\ell}\|_{L^\infty((\log R)^9\,Q_k)}\le (1+\delta)\|g_{k+1}\|_{L^\infty((\log R)^9 \,Q_k)}+(1+\delta)R^{-1000} .
\end{equation}
By the definition of $\Omega_{k}$, 
\[ \|g_{k+1}\|_{L^\infty((\log R)^9\,Q_k)}\le (1+\delta)^{N-k-1}r+(N-k-1)R^{-500}  . \]
This combined with (\ref{assump1}) gives
\[ \|g_k\|_{L^\infty((\log R)^9\,Q_k)}\le (1+\delta)^{N-k}r+(N-k)R^{-500},\]
here we used that $(1+\delta)(N-k-1)R^{-500}+(1+\delta)R^{-1000}\le (N-k)R^{-500}$ since $N\lesssim \log R/\log\log R$. This contradicts (\ref{B}) and means that (\ref{B'}) must be false, so the conclusion follows. 
\end{proof}

\vspace*{.2in}
\subsection{Proof of Proposition \ref{mainp}} 

Recall from Note \ref{normalization} that we have replaced $f$ by a constant multiple $cf$ so that $\underset{\theta}{\max}\|f_\theta\|_{L^\infty(\mathbb{R}^2)}= 1$.

The first step of the proof of Proposition \ref{mainp} involves an application of a local bilinear restriction theorem. We will use the following version. 

\begin{theorem}[Local bilinear restriction] \label{locbilTS} Let $\tau_k$ and $\tau_k'$ be nonadjacent $\sim R_k^{-1/2}\times R_k^{-1}$-rectangles in an $R_k^{-1}-$neighborhood of $\mb{P}^1$. Suppose $j\geq k$ and $f\in\mc{S}$ has Fourier support in $\mc{N}_{R_j^{-1}}(\mb{P}^1)$.  Suppose $T$ is in the range $R_j\geq T>  10 (\log R)  \, R_j^{1/2}/\text{dist}(\tau_k, \tau_k')$, and that $Q_T$ is a cube of sidelength $T$.  Then

\[ \int_{Q_T} |f_{\tau_k}f_{\tau_k'}|^2\lesssim \frac{(\log R)^4 T^{-2}}{\text{dist}(\tau_k, \tau_k') }\int \sum_{\tau_j\subset \tau_k} |f_{\tau_j}|^2 \w_{Q_T} \cdot \int \sum_{\tau_j'\subset \tau_k'} |f_{\tau_j}|^2 \w_{Q_T} +R^{-1000} \]
for a Gaussian weight function $\w_{Q_T}$ localized to $\log R \, Q_T$ and with Fourier transform essentially supported in a ball of radius $2\log R \,T^{-1}$. 
\end{theorem}

\begin{proof}
	Let $\phi_{Q_T} $be a Gaussian bump function adapted to $ Q_T$  as in Definition~\ref{Gaussian}. Then the Fourier transform $\widehat{\phi}_{Q_T}$ is essentially supported in a ball of radius $2\log R\,  T^{-1}$. 
	\begin{align*}
		\int_{Q_T} |f_{\tau_k} f_{\tau_k'}|^2 &\lesssim \int |\sum_{\tau_j\subset\tau_k, \tau_j'\subset\tau_k'} f_{\tau_j}f_{\tau_j'}|^2 \phi_{Q_T} \\
		&\lesssim \int \sum_{\tau_j\subset\tau_k, \tau_j'\subset\tau_k'} |f_{\tau_j}f_{\tau_j'}|^2 \phi_{Q_T}  + R^{-1000}
	\end{align*}
The reason for the last inequality is that for a fixed pair $(\tau_j, \tau_j')$, the number of pairs $(\tau_j'', \tau_j''')$ such that  
\begin{equation}\label{cordoba}
\tau_{j} + \tau_{j}'  + B_{2\log R T^{-1}} \bigcap \tau_j''+\tau_j''' +B_{2\log R T^{-1}} \neq \emptyset
\end{equation}
is at most $O(1)$. Here we used the fact that $T\geq 10 (\log R) \, R_{j}^{1/2}/\text{dist}(\tau_k, \tau_k')$. For more details checking \eqref{cordoba}, see the appendix.

It suffices to show that 
$$\int |f_{\tau_j}f_{\tau_j'}|^2 \phi_{Q_T}\lesssim (\log R)^2 T^{-2}\text{dist}(\tau_k, \tau_k') \int |f_{\tau_j}|^2 \w_{Q_T} \int |f_{\tau_j}|^2 \w_{Q_T}.$$

For a translate $T_{\tau_j}$ of $\tau_j^*$ and a translate   $T_{\tau_j}$ of $\tau_j'^{*}$,  let $c_{T_{\tau_j}}=\max_{x\in T_{\tau_j}} |f_{\tau_j}|^2(x)$ and $c_{T_{\tau_j'}}=\max_{x\in T_{\tau_j'}} |f_{\tau_j'}|^2(x)$.

Since $T_{\tau_j}$ and $T_{\tau_j'}$ have angle $\sim \text{dist}(\tau_k, \tau_k')$, we have $$|T_{\tau_j}\cap T_{\tau_j'}\cap \log R\, Q_T|\leq R_{j}/\text{dist}(\tau_k, \tau_k'),$$
while $T_{\tau_j}\cap \log R\, Q_T\sim R_{j}^{1/2} \log R \, T$. 
\begin{align*}
	\int |f_{\tau_j}f_{\tau_j'}|^2 \w_{Q_T}&\leq \int_{\log R\, Q_T} |f_{\tau_j} f_{\tau_j'}|^2 + R^{-1000}\\
	&\leq \int_{\log R\, Q_T} \sum_{T_{\tau_j}} \sum_{T_{\tau_j'}} c_{T_{\tau_j}}\chi_{T_{\tau_j}}\cdot  c_{T_{\tau_j'}}\chi_{T_{\tau_j'}}  +R^{-1000}\\
	&\lesssim \frac{(\log R)^2 T^{-2}}{ \text{dist}(\tau_k, \tau_k')} \int_{\log R \, Q_T} \sum_{T_{\tau_j}} c_{T_{\tau_j}} \chi_{T_{\tau_j}} \cdot \int_{\log R \, Q_T} \sum_{T_{\tau_j'}} c_{T_{\tau_j'}} \chi_{T_{\tau_j'}} +R^{-1000}
	\end{align*}

The next step is to show that 
\begin{equation}\label{eq: lcgaussian}
\int_{\log R Q_T} \sum_{T_{\tau_j}} c_{T_{\tau_j}}\chi_{T_{\tau_j}} \lesssim \int |f_{\tau_j}|^2 w_{Q_T} +R^{-1000}.
\end{equation}
The Fourier transform of $|f_{\tau_j}|^2$ is supported on $\tilde{\tau}_{j,0}=\tau_j+(-\tau_j)$, which is approximately an $R_j^{-1/2}\times R_j^{-1}$--rectangle. Let $\sum_{\tilde{\tau_j}} \phi_{\tilde{\tau_j}}$ be a Gaussian partition of unity for $\{\tilde{\tau_j}\}_{\tilde{\tau}_j\parallel \tilde{\tau}_{j,0}}$.  Let $\phi_{\tau_j} = \sum_{\tilde{\tau}_j\subset \log R \, \tilde{\tau}_{j,0}} \phi_{\tilde{\tau}_j}$. Then $|\phi_{\tau_j} - \chi_{\tilde{\tau}_{j,0}}|(\xi)\leq R^{-1000}$ for $\xi\in \tilde{\tau}_{j, 0}$ and $|f_{\tau_j}|^2 =|f_{\tau_j}|^2 \ast \widehat{\phi}_{\tau_j} +O(R^{-1000})$. 

Let $\psi_{\tau_j} (x) = \max_{y \in x+T_{\tau_j}-T_{\tau_j}} |\widehat{\phi}_{\tau_j}|$, then
$$\sum_{T_{\tau_j}} c_{T_{\tau_j}} \chi_{T_{\tau_j}} \leq |f_{\tau_j}|^2 \ast \psi_{\tau_j}.$$
 
We finish the proof since  $\psi_{\tau_j}\ast \chi_{\log R\, Q_T}\lesssim  (\log R)^2 w_{Q_T} +R^{-1000}$ for a Gaussian bump function $w_{Q_T}$ localized at $\log R \, Q_T$. 

\end{proof}

\begin{proof}[Proof of Proposition \ref{mainp}]

It suffices to bound $|U_\a\cap \Omega_k|$ for $k=1,\ldots,N-1$ and $|U_\a\cap L|$ since there are $\lesssim \log R$ of these sets and $|U_\a\cap Q_R|\le \sum\limits_{k=1}^{N-1}|U_\a\cap\Omega_k|+|U_\a\cap L|$.

For $k=1,\ldots,N-1$, by Lemma \ref{ftofk}, if $x\in U_\a\cap \Omega_k$, then
\[
 \max_{\tau, \tau' \textrm{ non-adj}}|f_\tau(x)f_{\tau'}(x)| \le  \max_{\tau, \tau' \textrm{ non-adj}}|f_{k+1,\tau}(x)f_{\tau'}(x)|+ (\log R)^{-10} \alpha |f_{\tau'}(x)| +R^{-500} ,
\]
where $f_{k+1,\tau}=\underset{\tau_k\subset\tau}{\sum}f_{k+1,\tau_k}$. By the definition of $U_\a$ (Definition \ref{defUa}), $|f_{\tau'}(x)| \le (\log R)^9 \alpha$, and so 
\begin{align*}
\max_{\tau, \tau' \textrm{ non-adj}}|f_\tau(x)f_{\tau'}(x)|\le & \max_{\tau, \tau' \textrm{ non-adj}}|f_{k+1,\tau}(x)f_{\tau'}(x)|+ (\log R)^{-1} \alpha^2 +R^{-500} 
\\
 \le &\max_{\tau,\tau' \textrm{ non-adj}} |f_{k+1, \tau}(x) f_{k+1, \tau'}(x)|
 + (\log R)^{-1} \alpha^2 \\
 &\qquad + (\log R)^{-10} \alpha |f_{k+1, \tau}(x)| + 2 R^{-500}. 
 \end{align*}

Using the definition of $U_\a$ as above as well as Lemma \ref{ftofk}, it follows that $|f_{k+1, \tau}(x)| \le 2 (\log R)^9 \alpha$, and so all together

$$ \max_{\tau, \tau' \textrm{ non-adj}}|f_\tau(x)f_{\tau'}(x)| \le  \max_{\tau, \tau' \textrm{ non-adj}}|f_{k+1,\tau}(x)f_{k+1, \tau'}(x)|+ 3 (\log R)^{-1} \alpha^2 + 2R^{-500} .$$

The $R^{-500}$ error term is negligible given our normalization, as explained in Note \ref{normalization}.  Since $x \in U_\a$, $ \max_{\tau, \tau' \textrm{ non-adj}}|f_\tau(x)f_{\tau'}(x)| \sim \a^2$, and so  $\max_{\tau, \tau' \textrm{ non-adj}}|f_{k+1,\tau}(x)f_{k+1, \tau'}(x)| \sim \a^2$ as well.  Therefore
\begin{equation}\label{eq: removek}
    \a^4|U_\a\cap\Omega_k| \lesssim \|\max_{\tau, \tau' \textrm{ non-adj}}|f_{k+1,\tau}f_{k+1,\tau'}|^{1/2}\|_{L^4(U_\a\cap\Omega_k)}^4
    \end{equation}
where $\tau$ and $\tau'$ are nonadjacent $\sim (\log R)^{-6}\times (\log R)^{-12}$ rectangles. We wish to apply the bilinear restriction theorem above, but the functions $f_{k+1,\tau_k}$ are only essentially supported in $\sim \tau_k$. This just means that we have an error term of $R^{-1000}$ which is negligible given our normalization, as explained in Note \ref{normalization}. 

For each $Q_k \subset \Omega_k$, 
\begin{align*}
\|\max_{\tau,\tau'}|f_{k+1,\tau}f_{k+1,\tau'}|^{1/2}\|_{L^4(Q_k)}^4&\le \sum_{\substack{\tau,\tau'\\\textrm{non-adj}}}\int_{Q_k}|f_{k+1,\tau}f_{k+1,\tau'}|^2\\
\text{(Theorem \ref{locbilTS}) }\qquad
    &\lesssim |Q_k|^{-1}\big(\int\sum_{\tau_k}|f_{k+1,\tau_k}|^2\w_{\tilde{Q}_k}\big)^2+R^{-1000}\\
    &\lesssim |Q_k|^{-1}\big(\int g_k\w_{\tilde{Q}_k}\big)^2+R^{-1000}|Q_k|+R^{-1000}
\end{align*}
where $w_{\tilde{Q}_k}$ is a weight function localized to $(\log R)^8Q_k$ and the final inequality follows from \eqref{eq: lc3201}  and \eqref{eq: lc3202} in the proof of Lemma \ref{hkbound}. 

By Lemma \ref{high} and the decay properties of $\w_{\tilde{Q}_k}$, 
\begin{align*} 
\int g_k\w_{\tilde{Q}_k}&\lesssim (\log R)^{16}\|g_k\|_{L^\infty((\log R)^9\,Q_k)}|Q_k|+R^{-500}\\
&\lesssim (\log R)^{17}\|g_k^h\|_{L^\infty((\log R)^9\,Q_k)}|Q_k|+R^{-500}\\
&\lesssim (\log R)^{18}\|g_k^h\|_{L^1(W_{Q_k})}+R^{-500}\\
&\lesssim (\log R)^{19}\|g_k^h\|_{L^2(W_{Q_k})}|Q_k|^{1/2}+R^{-500}
\end{align*}
where $W_{Q_k}$ is a Gaussian weight function localized to $\sim (\log R)^9 Q_k$  coming from the locally constant property (see \eqref{eq: lcgaussian}). Use this in the previous displayed math and add up the contributions from each $Q_k$ to obtain
\begin{equation}\label{here} \|\max_{\tau,\tau'}|f_{k,\tau}f_{k,\tau'}|^{1/2}\|_{L^4(\Omega_k)}^4\lesssim  (\log R)^{16}\int \big|g_k^h\big|^2\w_{\Omega_k} \end{equation}
where $\w_{\Omega_k}=\sum_{Q_k}W_{{Q}_k}$. Note that $\w_{\Omega_k}\lesssim 1$ and by the high lemma (Lemma \ref{hilem}),
\begin{align*} 
\int |g_k^h|^2\w_{\Omega_k} &\lesssim (\log R)^{\tilde{c}''}\int\sum_{\tau_k}|f_{k+1,\tau_k}|^4+R^{-1000} .
\end{align*} 

 Then $f_{k+1,\tau_k}=\underset{\tau_{k+1}\subset\tau_k}{\sum}f_{k+1,\tau_{k+1}}$, so 

$$\int\sum_{\tau_k}|f_{k+1,\tau_k}|^4 \le \frac{R_{k+1}^{3/2}}{R_{k}^{3/2}}\int\sum_{\tau_{k+1}}|f_{k+1,\tau_{k+1}}|^4. $$

By Lemma~\ref{propfk}, we have $\|  f_{k+1, \tau_{k+1}} \|_{L^\infty} \le C (\log R)^2 \lambda + R^{-1000}$.  The $R^{-1000}$ error term is negligible as explained in Note \ref{normalization}, and so we essentially have

$$\int\sum_{\tau_k}|f_{k+1,\tau_k}|^4  \le (\log R)^4 \lambda^2\int\sum_{\tau_{k+1}}|f_{k+1,\tau_{k+1}}|^2.$$

Finally, we have to carefully unwind the definition of $f_k$ to relate this last quantity to the original $f_\theta$:
\begin{equation}\label{eq: l2orth}
\int\sum_{\tau_{k+1}}|f_{k+1,\tau_{k+1}}|^2 \leq \sum_{\theta} \int |f_{\theta}|^2 + R^{-500}.
\end{equation}

 First we recall by Lemma \ref{propfk} that $|f_{k+1, \tau_{k+1}}(x)| \le |f_{k+2, \tau_{k+1}}(x)|$, and so
\begin{equation} \label{unwind1} \sum_{\tau_{k+1}}\int|f_{k+1,\tau_{k+1}}|^2 \le \sum_{\tau_{k+1}}\int|f_{k+2,\tau_{k+1}}|^2 \end{equation}
Next, by Definition \ref{fk}, $f_{k+2, \tau_{k+1}} = \sum_{\tau_{k+2} \subset \tau_{k+1}} f_{k+2, \tau_{k+2}}$, and so 
\begin{equation} \label{unwind2} \sum_{\tau_{k+1}}\int|f_{k+1,\tau_{k+1}}|^2\le \sum_{\tau_{k+1}}\int|\sum_{\tau_{k+2}\subset\tau_{k+1}}f_{k+2,\tau_{k+2}}|^2.  \end{equation}

By Lemma~\ref{propfk}, the Fourier transform of $f_{k+2,\tau_{k+2}}$ is essentially supported in $(1+(\log R)^{-8})\tau_{k+2}$. 
Since distinct $\tau_{k+2}$ and $\tau_{k+2}'$ are $\ge\frac{1}{2}R_{k+3}^{-1/2}$-separated, these sets are disjoint. By orthogonality, we get 
\begin{equation} \label{unwind3} \sum_{\tau_{k+1}}\int|\sum_{\tau_{k+2}\subset\tau_{k+1}}f_{k+2,\tau_{k+2}}|^2 \le \sum_{\tau_{k+2}}\int|f_{k+2,\tau_{k+2}}|^2  + R^{-500}. \end{equation}

Now we repeat the reasoning in inequalities \eqref{unwind1}, \eqref{unwind2}, and \eqref{unwind3} at many scales to conclude

\begin{align*}     
  \sum_{\tau_{k+1}}\int|\sum_{\tau_{k+2}\subset\tau_{k+1}}f_{k+2,\tau_{k+2}}|^2  &\le \sum_{\tau_{k+2}}\int|f_{k+2,\tau_{k+2}}|^2 \\
    &\le \sum_{\tau_{k+2}}\int|f_{k+3,\tau_{k+2}}|^2  \\
    &\le \sum_{\tau_{k+3}}\int|f_{k+3,\tau_{k+3}}|^2  \\
    &\cdots\\
    &\le \sum_\theta\int|f_{N,\theta}|^2  \\
    &\le \sum_{\theta}\int |f_\theta|^2. 
\end{align*}
In the above sequence of inequalities, we neglected to include an $R^{-500}$ added error term in each step due to the difference between ``essential support" and ``actual support."  These error terms are all negligible according to Note \ref{normalization}.  

The conclusion of this argument is that for $k=1,\ldots,N-1$,
\[\a^4|U_\a\cap \Omega_k|\lesssim (\log R)^{\tilde{c}'}\frac{r^2}{\a^2}\sum_\theta\|f_\theta\|_{L^2(\mathbb{R}^2)}^2 .\]

Finally we check that this indeed gives the conclusion of Proposition \ref{mainp}.  Recall that $g_N(x)=\sum_\theta|f_\theta|^2*\p_{\tilde{T}_\theta}$. By 
Lemma \ref{P}, $\| \p_{\tilde{T}_\theta} \|_{L^1} \le (\log R)^c$ . Thus for each $\theta$ and $x\in (\log R)^2 Q_N$ (where $Q_N\cap \Omega_N \neq \emptyset$), 
\[   |f_\theta|^2*\p_{\tilde{T}_\theta}(x)\lesssim (\log R)^{\tilde{c}}\|f_\theta\|_{L^\infty( \mathbb{R}^2 )}^2 +R^{-1000}. \]
It follows that $r\lesssim (\log R)^{\tilde{c}}\sum_\theta\|f_\theta\|_{L^\infty(\mathbb{R}^2)}^2 +R^{-1000}$. Plugging this in gives the conclusion of Proposition \ref{mainp}.  

Finally, it remains to bound $|U_\a \cap L|$. The first step is going from $f$ to $f_1$ using Lemma \ref{ftofk} (the argument for $\Omega_1$ in \eqref{eq: removek} holds for $L$ as well):
\begin{align*}
    \a^6|U_\a\cap L|&\lesssim  \int_{U_\a\cap L}\max_{\tau,\tau'}|f_{1,\tau}f_{1,\tau'}|^3+R^{-1000}\\
     & \lesssim (\log R)^{c} \int_{U_{\alpha}\cap L} ( \sum_{\tau_1} |f_{1, \tau_1}|^2)^3 +R^{-1000} \\
\text{( Lemma~\ref{propfk} )} \quad   & \lesssim (\log R)^{c} \int_{U_{\alpha}\cap L} ( \sum_{\tau_1} |f_{2, \tau_1}|^2)^3 +R^{-1000} 
    \\
    &\le  (\log R)^{c}  \int_{U_\a\cap L}g_1^2(\sum_{\tau_1} |f_{2, \tau_1}|^2)+R^{-1000}
\end{align*}
where  the last  inequality is due to \eqref{eq: lc3201}. Then by the definition of $L$, 
\[ \|g_1\|_{L^\infty(U_\a\cap L)}\lesssim Cr+NR^{-500}.  \]
Finally, by \eqref{eq: l2orth}, 
\[
\int \sum_{\tau_1} |f_{2, \tau_1}|^2 \leq \int \sum_{\theta} |f_{\theta}|^2 +R^{-500}.
\]

\end{proof}

\section{Proof of Theorem \ref{Main} -- the general case \label{remove}}

In the last section, we proved Proposition \ref{mainp}, which establishes our main theorem in the broad, well-spaced case.  In this section, we prove Theorem \ref{Main} in full generality.  We use Proposition \ref{mainp} as a black box, and then we remove the broad hypothesis by using a broad/narrow analysis, and we remove the well-spaced hypothesis by a random sampling argument.  

\subsection{Removing the broad hypothesis \label{remove1}}

The following proposition uses a broad/narrow analysis to prove an upper bound for $\|f\|_{L^6(X)}$ using Proposition \ref{mainp}.

\begin{prop}\label{mainstronger} There exist $c,C\in(0,\infty)$ such that for all well-spaced collections $\Theta$ and $f\in\mc{S}$ with Fourier support in $\underset{\theta\subset\Theta}{\cup}\theta$,
 \[  \|f\|_{L^6(X)}^6\le C(\log R)^c (\sum_{\theta}\|f_\theta\|_{L^\infty(\mathbb{R}^2)}^2)^2\sum_\theta \|f_\theta\|_{L^2(\mathbb{R}^2)}^2.\]
\end{prop}

First we prove a few technical lemmas. 

\begin{lemma}[Narrow lemma]\label{narrow}
Suppose that $\tilde{\tau}_k$ is an arc of length $R_k^{-1/2}\le \ell(\tilde{\tau}_k)\le 3R_k^{-1/2}$. Let $\{\tau_{k+1}\}$ be a partition of $\tilde{\tau}_k$ into $R_{k+1}^{-1/2}-$arcs. If $x$ satisfies
\begin{equation}\label{eq: hypox}
|f_{\tilde{\tau}_k}(x)|>\frac{(\log R)^2R_{k+1}^{1/2}}{R_k^{1/2}}\max_{\substack{\tau_{k+1},\tau_{k+1}'\\\textrm{ non-adj}}}|f_{\tau_{k+1}}(x)f_{\tau_{k+1}'}(x)|^{1/2}, 
\end{equation}
 then there exists an arc $\tilde{\tau}_{k+1}$ such that $\ell(\tilde{\tau}_{k+1})=3R_{k+1}^{-1/2}$ and
\[ |f_{\tilde{\tau}_k}(x)|\le (1+\frac{1}{\log R})|f_{\tilde{\tau}_{k+1}}(x)|.\]
\end{lemma}

\begin{proof} Write $f_{\tilde{\tau}_k}=\underset{\tau_{k+1}}{\sum} f_{\tau_{k+1}}$ and let $\tau_{k+1}^*$ index a summand satisfying
\[ \max_{\tau_{k+1}\subseteq \tilde{\tau}_k}|f_{\tau_{k+1}}(x)|=|f_{\tau_{k+1}^*}(x)|. \]
For each $\tau_{k+1}$ that is nonadjacent to $\tau_{k+1}^*$,
\[ |f_{\tau_{k+1}}(x)|\le |f_{\tau_{k+1}}(x)f_{\tau_{k+1}^*}(x)|^{1/2}<\frac{R_k^{1/2}}{(\log R)^2R_{k+1}^{1/2}}|f_{\tilde{\tau}_k}(x)| \]
using the hypothesis \eqref{eq: hypox} about $x$ in. Then
\[|f_{\tilde{\tau}_k}(x)-\sum_{\substack{{\tau_{k+1}}\,\,\textrm{non-adj}\\\text{to }{\tau_{k+1}^*}}} f_{\tau_{k+1}}(x)|>\big(1-\#\tau_{k+1}\frac{R_k^{1/2}}{(\log R)^2R_{k+1}^{1/2}}\big)|f(x)|. \]
The number of $\tau_{k+1}$ is bounded by $3R_{k+1}^{1/2}/R_k^{1/2}$. Define $\tilde{\tau}_{k+1}$ to be $(\tau_{k+1}^*)_L\cup\tau_{k+1}^*\cup(\tau_{k+1}^*)_R$ where $(\tau_{k+1}^*)_L$ is the left neighbor of $\tau_{k+1}^*$ and $(\tau_{k+1}^*)_R$ is the right neighbor of $\tau_{k+1}^*$.

\end{proof}

\begin{lemma}[Case 2 in the proof of Proposition~\ref{mainstronger}] \label{rescaling} \label{case2} Suppose that $\tau_k^*$ is an $\sim R_k^{-1/2}\times R_k^{-1}$-arc in a neighborhood of $\mb{P}^1$. Then \[ \int_{ H_{\tau_k^*}}|f_{\tau_k^*}|^6\lesssim (\log R)^c(\sum_{\theta\subset\tau_k^*}\|f_\theta\|_{L^\infty(\mathbb{R}^2)}^2)^2\sum_{\theta\subset\tau_k^*}\|f_\theta\|_{L^2(\mathbb{R}^2)}^2  \]
where
\begin{align*} 
H_{\tau_k^*}&=\{x\in Q_R:|f_{\tau_k^*}(x)|\le  (\log R)^8\underset{\substack{\tau_{k+1},\tau_{k+1}'\\ \textrm{non-adj}}}{\max}|f_{\tau_{k+1}}(x)f_{\tau_{k+1}'}(x)|^{1/2}\\
&\qquad \text{and}\quad (\log R)^{-9}(\sum_{\tau_{k+1}\subset\tau_k^*}|f_{\tau_{k+1}}(x)|^6)^{1/6}\le \max_{\substack{\tau_{k+1},\tau_{k+1}'\subset\tau_k^* \\ \textrm{non-adj}}} |f_{\tau_{k+1}}(x)f_{\tau_{k+1}'}(x)|^{1/2}  \}.  \end{align*}
Here $\tau_{k+1}$ are $R_{k+1}^{-1/2}\times R_{k+1}^{-1}$ rectangles.
\end{lemma}

\begin{proof} Let $(c,c^2)$ be the center of $\tau_k^*\cap \mathbb{P}^1$. Define the affine map
\begin{equation}\label{affine map} \ell(\xi_1,\xi_2)=\big(R_k^{1/2}(\xi_1-c),R_k(\xi_2-2\xi_1c+c^2)\big). 
\end{equation}
Then $\ell(\tau_k^*)$ is contained in an $(R/R_k)^{-1}$-neighborhood of $\mb{P}^1$. The images $\{\ell(\theta)\}_{\theta\subset\tau_k^*}$ have the spacing property at scales $R_{k+1}/R_k,\ldots,R/R_k$. Define the function $h$ as 
\[ \widehat{h}=\widehat{f}\circ\ell^{-1} \]
and note that for each $R_l^{-1/2}\times R_l^{-1}-$rectangle $\tau_l\subset\tau_k^*$,
\[ R_k^{-3/2}h_{\ell(\tau_l)}(\g(x))e^{2\pi i x\cdot(c,c^2)}=f_{\tau_l}(x) \]
where $\ell(\tau_l)$ is approximately a $(R_l/R_k)^{-1/2}\times (R_l/R_k)^{-1}$-rectangle and
\[ \g(x)=\big(\frac{x_1+2cx_2}{R_k^{1/2}},\frac{x_2}{R_k}\big). \] 
In particular,
\begin{equation}\label{particular}\int_{ H_{\tau_k^*}}|f_{\tau_k^*}|^6=R_k^{-9+3/2}\int_{\g( H_{\tau_k^*})}|h_{\ell(\tau_k^*)}(x)|^6. \end{equation}

By dyadic pigeonholing, there exists $\a_{\tau_k^*}>0$ such that
\[ \|h_{\ell(\tau_k^*)}\|_{L^6(\g(H_{\tau_k^*}))}^6\lessapprox \a_{\tau_k^*}^6|\{x\in \g( H_{\tau_k^*}):\max_{\substack{\tau_{k+1},\tau_{k+1}'\subset\tau_k^* \\\textrm{non-adj}}}|h_{\ell(\tau_{k+1})}(x)h_{\ell(\tau_{k+1}')}(x)|^{1/2}\sim\a_{\tau_k^*}\}|.  \]
Repeat the proof of Proposition \ref{mainp} to obtain
\begin{equation}\label{particularii} \|h_{\ell(\tau_k^*)}\|_{L^6(\g(H_{\tau_k^*}))}^6\lesssim (\log R)^c(\sum_{\theta\subset\tau_k^*}\|h_{\ell(\theta)}\|_{L^\infty(\mathbb{R}^2)}^2)^2\sum_{\theta\subset\tau_k^*}\|h_{\ell(\theta)}\|_{L^2(\mathbb{R}^2)}^2. \end{equation}

First observe that 
\[ \sum_{\theta\subset\tau_k^*}\|h_{\ell(\theta)}\|_{L^2(\mathbb{R}^2)}^2\lesssim R_k^{3-3/2}\sum_{\theta\subset\tau_k^*}\|f_\theta\|_{L^2(\mathbb{R}^2)}^2 .\]
Next, note that 
for each $\theta\subset\tau_k$
\[ \|h_{\ell(\theta)}\|_{L^\infty(\mathbb{R}^2)}^2\le R_k^{3}\|f_\theta\|_{L^\infty(\mathbb{R}^2)}^2. \]
These observations combined with (\ref{particular}) and (\ref{particularii}) give the desired conclusion. 
\end{proof}

\begin{proof}[Proof of Proposition \ref{mainstronger}] Define an iteration using a broad/narrow argument. 

\noindent {\bf{Initial step:}}  Define $S_1$ to be the set 
\begin{align} \label{eq: defS}
S_1:=\{x\in X:& |f(x)|\le (\log R)^{8} \max_{\tau,\tau'\, \textrm{non-adj}}|f_{\tau}(x)f_{\tau'}(x)|^{1/2}\quad \text{and}\\
&\quad (\sum_{\tau}|f_{\tau}(x)|^6)^{1/6}\le (\log R)^{9} \max_{\tau,\tau'\, \textrm{non-adj}}|f_{\tau}(x)f_{\tau'}(x)|^{1/2} \}.\nonumber
\end{align} 
 Define $B_1=X\setminus S_1$. Split the integral into 
\begin{align}
    \int_{X}|f|^6&= \int_{S_1}|f|^6+\int_{B_1}|f|^6. 
\end{align}
By the narrow lemma, if $x\in B_1$ satisfies $|f(x)|> (\log R)^8\underset{\tau,\tau' \, \textrm{non-adj}}{\max}|f_{\tau}(x)f_{\tau'}(x)|^{1/2}$, then for $\{\tau^{**}\}$ a collection of pairwise disjoint unions of three consecutive $\tau$, 
\[ |f(x)|\le (1+\frac{1}{\log R})\big(\sum_{\tau^{**}}|f_{\tau^{**}}(x)|^6\big)^{1/6}. \]
Alternatively, $x\in B_1$ satisfies 
\[|f(x)|\le  (\log R)^8\underset{\tau,\tau', \textrm{non-adj}}{\max}|f_{\tau}(x)f_{\tau'}(x)|^{1/2}\]
but 
\[ (\log R)^{-9}(\sum_{\tau}|f_{\tau}(x)|^6)^{1/6}>\underset{\tau,\tau'\, \textrm{non-adj}}{\max}|f_{\tau}(x)f_{\tau'}(x)|^{1/2}    . \] 
Putting this together means that
\begin{align*}  
\int_{B_1}|f|^6 &\le (1+\frac{1}{\log R})^6\int_{B_1}\sum_{\tau^{**}}|f_{\tau^{**}}|^6 +\frac{1}{(\log R)^{6}} \int_{ B_1}\sum_{\tau}|f_{\tau}|^6. \end{align*} 
Let $\{\tau^{*}\}$ denote the collection $\{\tau\}$ if  
\[ \int_{ B_1}\sum_{\tau}|f_{\tau}|^6\ge \int_{B_1}\sum_{\tau^{**}}|f_{\tau^{**}}|^6   \]
and equal $\{\tau^{**}\}$ otherwise. Then
\[\int_{B_1}|f_\tau|^6 \le (1+\frac{2}{\log R})^6\int_{ B_1}\sum_{\tau^{*}}|f_{\tau^{*}}|^6   \]
(this just means we only have one finer scale to keep track of rather than two almost equivalent scales). Summarizing all of the inequalities, conclude that
\begin{align*} 
\|f\|_{L^6(X)}&\le \int_{S_1}|f|^6+(1+\frac{2}{\log R})^6\sum_{\tau^*}\int_{B(1)} |f_{\tau^*}|^6
\end{align*} 
For each $\tau^*$, further decompose $B_1$ into 
\begin{align*} 
S_{\tau^*}=\{x\in B_1:\,\,&|f_{\tau^{*}}(x)|\le (\log R)^8 \max_{ \tau_{2},\tau_{2}'\subset\tau^*\,\textrm{non-adj}}|f_{\tau_{2}}(x)f_{\tau_{2}'}(x)|^{1/2}\quad \text{and}\\
&\quad (\sum_{\tau_{2}\subset\tau^*}|f_{\tau_{2}}(x)|^6)^{1/6}\le (\log R)^9 \max_{\tau_{2},\tau_{2}'\subset\tau^*\,\textrm{non-adj}}|f_{\tau_{2}}(x)f_{\tau_{2}'}(x)|^{1/2}  \} 
\end{align*} 
where $\ell(\tau_2)= R_{2}^{-1/2}$. By analogous reasoning as above, conclude this case with the inequality
\begin{align*}
    \|f\|_{L^6(X)}^6&\le \int_{S_1}|f|^6+(1+\frac{2}{\log R})^6\sum_{\tau^*}\int_{S_{\tau^*}} |f_{\tau^*}|^6+(1+\frac{2}{\log R})^{12}\sum_{\tau^*}\int_{B_{\tau^*}}\sum_{\tau_2^*\subset\tau}|f_{\tau_2}|^6
\end{align*} 
where $B_{\tau^*}=B_1\setminus S_{\tau^*}$.

\noindent {\bf{Step k:}} ($k\ge 2$)  The conclusion of the previous step is 
\begin{align} \label{eq: brnarrow}
\|f\|_{L^6(X)}^6&\le \int_{S_1}|f|^6+(1+\frac{2}{\log R})^6\sum_{\tau^*}\int_{S_{\tau^*}}|f_{\tau^*}|^6 +\cdots+(1+\frac{2}{\log R})^{6(k-1)}\sum_{\tau_{k-1}^{*}}\int_{S_{\tau_{k-1}^{*}}}|f_{\tau_{k-1}^{*}}|^6\\
&\qquad+(1+\frac{2}{\log R})^{6k}\sum_{\tau_{k-1}^{*}}\int_{B_{\tau_{k-1}^{*}}}\sum_{\tau_k^{*}\subset\tau_{k-1}^*}|f_{\tau_k^{*}}|^6  \nonumber
\end{align} 
where for each $\tau_{k-1}^{*}$, if $\tau_{k-1}^{*}\subset\tau_{k-2}^{*}\subset\cdots\subset\tau_2^{*}\subset \tau$, 
\[ B_{\tau_{k-1}^{*}}=B_1\setminus (S_1\cup S_{\tau^*}\cup S_{\tau_2^{*}}\cup\cdots\cup S_{\tau_{k-1}^{*}}). \]

For each $\tau_k^{*}\subset \tau_{k-1}^*$, define $S_{\tau_k^{*}}$ to be the set
\begin{align*}  
\{x\in B_{\tau_{k-1}^{*}}:&|f_{\tau_k^{*}}(x)|\le (\log R)^8 \max_{\substack{\tau_{k+1},\tau_{k+1}'\subset\tau_k^*\\\textrm{non-adj}}}|f_{\tau_{k+1}}(x)f_{\tau_{k+1}'}(x)|^{1/2}\quad \text{and}\\
&\quad (\sum_{\tau_{k+1}\subset\tau_k^*}|f_{\tau_{k+1}}(x)|^6)^{1/6}\le (\log R)^9 \max_{\substack{\tau_{k+1},\tau_{k+1}'\subset\tau_k^* \\ \textrm{non-adj}}}|f_{\tau_{k+1}}(x)f_{\tau_{k+1}'}(x)|^{1/2}  \} 
\end{align*} 
where $\ell(\tau_{k+1})= R_{k+1}^{-1/2}$. Define $B_{\tau_k^*}=B_{\tau_{k-1}^*}\setminus S_{\tau_k^*}$. By analogous arguments as above, conclude that \eqref{eq: brnarrow} holds with $k$ replaced by $k+1$. 

\vspace*{.2in}
Iterate this procedure until Step $N-1$ where $N\sim \frac{\log R}{\log\log R}$ to obtain  inequality \eqref{eq: brnarrow} for $k=n$.
Since there are $\sim(\log R)/(\log\log R)$ terms in the right hand side, it suffices to consider cases where $\|f\|_{L^6(Q_R)}^6$ is bounded by $(\log R)$ times one of the terms on the right hand side. Note that the factors 
$(1+\frac{2}{\log R})^{6k}$ are $\lesssim 1$ for all $k\le N$. 
\[\]
\noindent {\bf{Case 1:}}
\[ \|f\|_{L^6(Q_R)}^6\lesssim (\log R)\int_{S_1}|f|^6 .\]
By the definition of $S_1$ in \eqref{eq: defS}, 
\[ \int_{S_1}|f|^6\le (\log R)^{48}\int_{S_1}\max_{\tau,\tau'}|f_{\tau}f_{\tau'}|^3 . \]
Let $U_s=\{x\in S_1:\max_{\tau,\tau'}|f_{\tau}(x)f_{\tau'}(x)|^{1/2}\le R^{-10}\max_{\theta}\|f_{\theta}\|_{L^{\infty}}\}$. Note that
\begin{align*}  
\int_{U_s}\max_{\tau,\tau'}|f_{\tau}f_{\tau'}|^3&\le  R^{-55} \max_{\theta}\|f_{\theta}\|_{L^{\infty}(\mathbb{R}^2)}\\
\text{(Lemma~\ref{lem: lc})\quad }&\le R^{-55}(\sum_\theta\|f_\theta\|_{L^\infty(\mathbb{R}^2)}^2)^2\int\sum_\theta|f_\theta|^2,
\end{align*} 
which is the right hand side of Proposition \ref{mainstronger}. 

Then $S_1\setminus U_s$ can be partitioned into $\lesssim \log R$ sets $U_\a$ on which $\underset{\tau,\tau'}{\max}|f_\tau f_{\tau'}|^{1/2}\sim \a$ with $R^{-10}\leq \alpha / \max_{\theta}\|f_{\theta}\|_{L^{\infty}(\mathbb{R}^2)}  \leq R$.  By pigeonholing,  
\[ \int_{S_1\setminus U_s}\max_{\tau,\tau'}|f_{\tau}f_{\tau'}|^3\lesssim (\log R)\int_{S_1\cap U_\a}\max_{\tau,\tau'}|f_\tau f_{\tau'}|^3\sim (\log R)\a^6|S_1\cap U_\a|.    \]
Then Proposition \ref{mainp} applies to bound $\a^6|S_1\cap U_\a|$.

\noindent {\bf{Case 2:}}
\[ \|f\|_{L^6(X)}^6\lesssim (\log R)\sum_{\tau_k^*}\int_{S_{\tau_k}^*}|f_{\tau_k^*}|^6    .\]
The sets $S_{\tau_k^*}$ are contained in $H_{\tau_k^*}$ from Lemma \ref{case2}. Using Lemma \ref{case2},
\begin{align*} 
\sum_{\tau_k^*}\int_{S_{\tau_k}^*}|f_{\tau_k^*}|^6  &\lesssim \sum_{\tau_k^*}(\log R)^c(\sum_{\theta\subset\tau_k^*}\|f_\theta\|_{L^\infty(\mathbb{R}^2)}^2)^2\sum_{\theta\subset\tau_k^*}\|f_\theta\|_{L^2(\mathbb{R}^2)}^2\\
&\le (\log R)^c(\sum_{\theta}\|f_\theta\|_{L^\infty(\mathbb{R}^2)}^2)^2\sum_{\theta}\|f_\theta\|_{L^2(\mathbb{R}^2)}^2,
\end{align*} 
so we have the desired conclusion.

\noindent {\bf{Case 3:}}
\[ \|f\|_{L^6(X)}^6 \lesssim (\log R)\sum_{\tau_{N-1}^*}\int_{B_{\tau_{N-1}^*}}\sum_{\tau_N^*\subset\tau_{N-1}^*}|f_{\tau_N^*}|^6.\] 
The size of the $\tau_N^*$ is approximately the size of the $\theta$ (up to a factor of $3$), so we may assume $\tau_N^*=\theta$. Then since the $\ell^6$ norm is bounded by the $\ell^2$ norm, the above inequality implies that 
\[ \a^6|U_\a\cap X|\lesssim (\log R)\int_{U_\a\cap X}(\sum_\theta|f_\theta|^2)^3\lesssim (\log R)(\sum_\theta\|f_\theta\|_{L^\infty(\mathbb{R}^2)}^2)^2\sum_\theta\|f_\theta\|_{L^2(\mathbb{R}^2)}^2. \]

\end{proof}

\vspace*{.2in}

\vspace*{.2in}

\subsection{Removing the well-spaced hypothesis}
 \vspace*{.2in}

This section contains two well-spaced lemmas which we will use to prove Theorem \ref{Main} in the following section. 
\vspace*{.2in}

\begin{lemma} \label{scaleRk} For each $f\in\mc{S}$ with Fourier transform supported in $\mc{N}_{R^{-1}}(\mb{P}^1)$, there exists a collection $\Theta_{N-1}$ that is well-spaced at scale $R_{N-1}$ such that 
\[(1-\frac{1}{\log R})^6\frac{\|f\|_{L^6(Q_R)}^6}{\underset{\theta}{\sum}\|f_\theta\|_{L^2(\mathbb{R}^2)}^2}\le \frac{\|\tilde{f}\|_{L^6(Q_R)}^6}{\underset{\theta}{\sum}\|(\tilde{f})_\theta\|_{L^2(\mathbb{R}^2)}^2} \]
where $\tilde{f}=\underset{\theta\in\Theta_{N-1}}{\sum}f_\theta$. 
\end{lemma}

\begin{proof} Write $\|\cdot\|_2$ to denote $\|\cdot\|_{L^2(\mathbb{R}^2)}$. Label the $R^{-1/2}$-arcs from left to right by $\theta_1,\ldots,\theta_n$. For each $i=0,\ldots,(\log R)^6$, define $H^i$ by 
\[ \sum_{m\in \mathbb{Z}: \, 1\le i+m(\log R)^6 \le n} f_{\theta_{i+m(\log R)^6}}  .  \]
Then $f=\underset{i}{\sum}H^i$ where there are $R^{1/2}/R_{N-1}^{1/2}=(\log R)^6$ terms in the sum. 
Additionally, 
\[ \sum_{\theta}\|f_\theta\|_{2}^2=\sum_i\sum_\theta\|(H^i)_\theta\|_{2}^2. \]
Let $i_0$ denote an index satisfying
\[ \sum_\theta\|f_\theta\|_{2}^2\ge (\log R)^6\sum_\theta\|(H^{i_0})_\theta\|_{2}^2. \]
Define $F$ by 
\[ f=F+H^{i_0}.   \]
Both $F$ and $H^{i_0}$ have the spacing property at scale $R_{N-1}$. 
\begin{itemize}
    \item If $\|H^{i_0}\|_{L^6(Q_R)}\le\frac{1}{\log R}\|f\|_{L^6(Q_R)}$, then define $\tilde{f}=F$. Note that 
    \[ (1-\frac{1}{\log R})\|f\|_{L^6(Q_R)}\le\|\tilde{f}\|_{L^6(Q_R)}.\]

    \item If $\|H^{i_0}\|_{L^6(Q_R)}>\frac{1}{\log R}\|f\|_{L^6(Q_R)}$, then define $\tilde{f}=H^{i_0}$. From the way we selected $i_0$, we have
    \[ \frac{\|f\|_{L^6(Q_R)}^6}{\underset{\theta}{\sum}\|f_\theta\|_{L^2(\mathbb{R}^2)}^2}\le \frac{(\log R)^6\|\tilde{f}\|_{L^6(Q_R)}^6}{(\log R)^6\underset{\theta}{\sum}\|(\tilde{f})_\theta\|_{L^2(\mathbb{R}^2)}^2} . \]
\end{itemize}
No matter the case above, since for each $\theta$, $f_\theta=F_\theta$ or $H_\theta^{i_0}$,
\[ \sum_\theta\|(\tilde{f})_\theta\|_{2}^2\le \sum_\theta\|f_\theta \|_{2}^2.    \]

\end{proof}

\vspace*{.2in}

\begin{lemma} \label{well-spaced} For each $f\in\mc{S}$ with Fourier transform supported in $\mc{N}_{R^{-1}}(\mb{P}^1)$, there exists a well-spaced collection $\tilde{\Theta}$ such that 
\[ \frac{\|f\|_{L^6(Q_R)}^6}{\underset{\theta}{\sum}\|f_\theta\|_{L^2(\mathbb{R}^2)}^2}   \lesssim   \,\, \frac{\|\tilde{f}\|_{L^6(Q_R)}^6}{\underset{\theta}{\sum}\|(\tilde{f})_\theta\|_{L^2(\mathbb{R}^2)}^2}     \]
where $\tilde{f}=\underset{\theta\in\tilde{\Theta}}{\sum}f_\theta$. 
\end{lemma}

\begin{proof} Define an iterative procedure. Let $\tilde{F}_N=f$. Apply Proposition \ref{scaleRk} to $\tilde{F}_N$ to obtain $\tilde{F}_{N-1}$. 
\[\]
\noindent {\bf{Obtaining $\tilde{F}_{N-k-1}$ (where $k\ge 2$) from $\tilde{F}_{N-k}$:}}
We have $\tilde{F}_{N-k}$ from the previous step. $\tilde{F}_{N-k}$ has the spacing property at scales $R_{N-1},\ldots,R_{N-k}$. Furthermore,  
\begin{equation}\label{eq: lem45}
(1-\frac{1}{\log R})^{6k}\frac{\|f\|_{L^6(Q_R)}^6}{\underset{\theta}{\sum}\|f_\theta\|_{2}^2}\le \frac{\|\tilde{F}_{N-k}\|_{L^6(Q_R)}^6}{\underset{\theta}{\sum}\|(\tilde{F}_{N-k})_\theta\|_2^2}.  
\end{equation}
The $R_{N-k}^{-1/2}$-arcs made up of  unions of the $\theta$ are defined by the previous steps. Label them from left to right by $\tau_{1},\ldots,\tau_{n_k}$. For $i=0,\ldots,(\log R)^6$, define $H_{N-k-1}^i$ by 
\[ \sum_{1\le i+ m(\log R)^6\le n_k} (\tilde{F}_{N-k})_{\tau_{i+m(\log R)^6}}  . \]
Note that $\tilde{F}_{N-k}=\underset{i}{\sum}H_{N-k-1}^i$ and for each $\theta$, $(\tilde{F}_{N-k})_\theta$ equals $(H_{N-k-1}^i)_\theta$ for exactly one $i$. Let $i_k$ denote the index satisfying 
\[ \sum_\theta\|(\tilde{F}_{N-k})_\theta\|_{L^6(Q_R)}^2\ge (\log R)^6\sum_\theta\|(H_{N-k-1}^{i_k})_\theta\|_{L^6(Q_R)}^2. \]

Define $F_{N-k-1}$ by 
\[ \tilde{F}_{N-k}=F_{N-k-1}+H_{N-k-1}^{i_k}.   \]
Note that since both $F_{N-k-1}$ and $H_{N-k-1}^{i_k}$ have smaller Fourier support than $\tilde{F}_{N-k}$, they inherit the spacing property at scales $R_{N-1},\ldots,R_{N-k}$. By construction, $F_{N-k-1}$ and $H_{N-k-1}^{i_k}$ also have the spacing property at scale $R_{N-k-1}$. 
\[\]
Define $\tilde{F}_{N-k-1}$ as follows: 
\begin{itemize}
    \item If $\|H_{N-k-1}^{i_k}\|_{L^6(Q_R)}\le\frac{1}{\log R}\|\tilde{F}_{N-k}\|_{L^6(Q_R)}$, then define $\tilde{F}_{N-k-1}=F_{N-k-1}$ and note that \eqref{eq: lem45} holds with $k$ replaced by $k+1$.

    \item If $\|H_{N-k-1}^{i_k}\|_{L^6(Q_r)}>\frac{1}{\log R}\|\tilde{F}_{N-k}\|_{L^6(Q_R)}$, then define $\tilde{F}_{N-k-1}=H_{N-k-1}^{i_k}$ and note that because of the way we defined $H_{N-k-1}^{i_k}$,
    \[ \frac{\|\tilde{F}_{N-k}\|_{L^6(Q_R)}^6}{\underset{\theta}{\sum}\|(\tilde{F}_{N-k})_\theta\|_{2}^2} \le   \frac{\|\tilde{F}_{N-k-1}\|_{L^6(Q_R)}^6}{\underset{\theta}{\sum}\|(\tilde{F}_{N-k-1})_\theta\|_{2}^2}.     \]
\end{itemize}

Iterate this procedure for $N-1$-steps, until we obtain $\tilde{F}_1$ which is well-spaced along with  inequality \eqref{eq: lem45} holds for $k=N-1$.
Since $N\sim \frac{\log R}{\log\log R}$, $(1-\frac{1}{\log R})^{-N}\le e^{\frac{C}{\log\log R}}\lesssim 1 $

\end{proof}

\subsection{Proof of Theorem \ref{Main}}

We prove Theorem \ref{Main} using Proposition \ref{mainstronger}  and Lemma \ref{well-spaced}. 

\begin{proof}[Proof of Theorem \ref{Main}]
By Lemma \ref{well-spaced}, 
\[ \frac{\|f\|_{L^6(Q_R)}^6}{\underset{\theta}{\sum}\|f_\theta\|_{L^2(\mathbb{R}^2)}^2}   \lesssim   \,\, \frac{\|\tilde{f}\|_{L^6(Q_R)}^6}{\underset{\theta}{\sum}\|(\tilde{f})_\theta\|_{L^2(\mathbb{R}^2)}^2}     \]
where $\tilde{f}=\underset{\theta\in\tilde{\Theta}}{\sum}f_\theta$ for a well-spaced $\tilde{\Theta}$. Then by Proposition \ref{mainstronger}, 
\[  \frac{\|\tilde{f}\|_{L^6(Q_R)}^6}{\underset{\theta}{\sum}\|(\tilde{f})_\theta\|_{L^2(\mathbb{R}^2)}^2}\lesssim (\log R)^c(\sum_{\theta\in\tilde{\Theta}}\|f_\theta\|_{L^\infty(\mathbb{R}^2)}^2)^2.   \]
Since $\tilde{\Theta}\subset\Theta$, we are done. 
\end{proof}

\section{\label{pigeonholing} Showing $\text{Dec}_6(R)\lesssim (\log R)^{c}$ from Theorem \ref{Main}}

\subsection{Wave packet decomposition and pigeonholing}

We will consider the following form of the decoupling inequality:
\begin{equation} \label{dec} \|f\|_{L^6(Q_R)}\le \text{Dec}_6(R)\big(\sum_\theta\|f_\theta\|_{L^6(\R^2)}^2\big)^{1/2}.\end{equation}
The constant $\text{Dec}_6(R)$ associated to this inequality is comparable to the constant where the $L^6$-norms are both taken over $\R^2$ and to the constant obtained from the $L^6$ norm in the upper bound being some weight function $\w_{Q_R}$.

Our goal is to begin with $f\in\mc{S}$ with Fourier transform supported in $\mc{N}_{R^{-1}}(\mb{P}^1)$ and show it suffices to prove the decoupling inequality for a version of $f$ which has relatively constant amplitudes and number of wave packets in each direction. In the following definitions, $\sim$ means within a factor of $2$.

Write 
\begin{align}\label{sum} f=\sum_\theta\sum_{T\in\T_\theta}\s_Tf_\theta \end{align}
where for each $\theta$, $\{\s_T\}_{T\in\T_\theta}$ is a Gaussian partition of unity (meaning adds up to 1) adapted to $(\log R)^9(R\times R^{1/2})-$tubes $T$. Note that this implies that $|\widehat{\s_T}|\lesssim R^{-1000}$ off of $(\log R)^{-3}(\theta-c_\theta)$, where $c_\theta$ is the center of $\theta$.

\begin{prop}[Wave packet decomposition] \label{wpd} There exist subsets $\tilde{\Theta}\subset\Theta$ and $\tilde{\T}_\theta\subset\T_\theta$ as well as a constant $C\in[R^{-10^3},1]$ with the following properties:
\begin{align} \|\sum_\theta f_\theta\|_{L^6(Q_R)}\lesssim (\log R)^2 \|\sum_{\theta\in\tilde{\Theta}}&\sum_{T\in\tilde{\T}_\theta}\s_Tf_\theta\|_{L^6(Q_R)} +R^{-9.5}\big(\sum_\theta\|f_\theta\|_{L^6(\R^2)}^2\big)^{1/2}, \\
\#\tilde{\T}_\theta\sim \#\tilde{\T}_{\theta'}&\quad \text{ for all }\quad \theta,\theta'\in\tilde{\Theta}, \\
\text{and} \quad \|\s_Tf_\theta\|_{L^\infty(\R^2)}\sim C\max_{\theta'\in \tilde{\Theta}}&\max_{T'\in\tilde{\T}_{\theta'}}\|\s_{T'}f_{\theta'}\|_{L^\infty(\R^2)} \quad\text{for all}\quad \theta\in\tilde{\Theta}\,\,\text{and}\,\,  T\in\tilde{\T}_\theta. 
\end{align}
\end{prop}

\begin{proof} 
Split the sum (\ref{sum}) into 
\begin{equation}\label{step1} f=\sum_\theta\sum_{T\in\T_\theta^c}\s_Tf_\theta+\sum_\theta\sum_{T\in\T_\theta^f}\s_Tf_\theta\end{equation} 
where the close set is
\[ \T_\theta^c:=\{T\in\T_\theta:T\cap R^{10} Q_R\not=\emptyset\}\]
and the far set is 
\[ \T_\theta^f:=\{T\in\T_\theta:T\cap R^{10}Q_R=\emptyset\} . \]
Using Cauchy-Schwarz, 
\begin{align*}
\|\sum_\theta\sum_{T\in\T_\theta^f}\s_Tf_\theta\|_{L^6(Q_R)}&\le R^{1/2}\|(\sum_\theta|\sum_{T\in\T_\theta^f}\s_Tf_\theta|^2)^{1/2}\|_{L^6(Q_R)}     \\
&\le R^{1/2}\big(\sum_\theta\|\sum_{T\in\T_\theta^f}\s_Tf_\theta\|_{L^6(Q_R)}^2\big)^{1/2} \\
&\le R^{1/2}\max_\theta\|\sum_{T\in\T_\theta^f}\s_T\|_{L^\infty(Q_R)}\big(\sum_\theta\|f_\theta\|_{L^6(\R^2)}^2\big)^{1/2} \\
&\le \frac{1}{R^{9.5}}\big(\sum_\theta\|f_\theta\|_{L^6(\R^2)}^2\big)^{1/2} 
\end{align*} 
where we could have bounded the $L^\infty$ norm by $C_NR^{-10N}$ for any $N\in\N$. This takes care of the \emph{far} portion of $f$ (i.e. the second term on the right hand side of (\ref{step1})).

The close set  has cardinality $|\T_{\theta}^c|\leq R^{22}$.  Let $M=\max_\theta\max_{T\in\T_\theta^c}\|\s_Tf_\theta\|_{L^\infty(\R^2)}.$ By Lemma~\ref{lem: lc}, 
\begin{equation}\label{eq: M}
M \leq \max_{\theta}\|f_{\theta}\|_{L^{\infty}} \lesssim (\sum_{\theta}\|f_{\theta}\|_{L^6(\mathbb{R}^2)}^2)^{1/2}.
\end{equation}

Split the remaining term as 
\begin{equation} \label{step2}
    \sum_\theta\sum_{T\in\T_\theta^c}\s_Tf_\theta=\sum_\theta\sum_{R^{-10^3}\le \lambda\le 1}\sum_{T\in\T_{\theta,\lambda}^c}\s_Tf_\theta+\sum_\theta\sum_{T\in\T_{\theta,s}^c}\s_Tf_\theta
\end{equation}
where $\lambda$ is a dyadic number in the range $[R^{-10^3},1]$,  
\[ \T_{\theta,\lambda}^c:=\{T\in\T_\theta^c:\|\s_Tf_\theta\|_{L^\infty(\R^2)}\sim \lambda M \},\]
and
\[ \T_{\theta,s}^c:= \{T\in\T_\theta^c:\|\s_Tf_\theta\|_{L^\infty(\R^2)}\le \frac{1}{2}R^{-10^3}M \} . \]

Handle the \emph{small} term from (\ref{step2}) by
\begin{align*}   \|\sum_\theta\sum_{T\in\T_{\theta,s}^c}\s_Tf_\theta\|_{L^6(Q_R)} &\le R^{1/2}\big(\sum_\theta\|\sum_{T\in\T_{\theta,s}^c}\s_Tf_\theta\|_{L^6(Q_R)}^2\big)^{1/2}   \\ &\leq R^{-10}M \leq R^{-10} (\sum_{\theta}\|f_{\theta}\|_{L^6(\mathbb{R}^2)}^2)^{1/2}. 
\end{align*} 

Next decompose the remaining term from (\ref{step2}) as
\begin{equation} \label{step3}
\sum_{R^{-10^3}\le\lambda\le 1}\sum_\theta\sum_{T\in\T_{\theta,\lambda}^c}\s_Tf_\theta=\sum_{R^{-10^3}\le \lambda\le 1}\sum_{1\le j\le R^{22}}\sum_{\theta\in\Theta_j(\lambda)}\sum_{T\in\T_{\theta,\lambda}^c}\s_Tf_\theta 
\end{equation}
where $j$ is a dyadic number in the range $[1,R^{22}]$ and $\Theta_j(\lambda)=\{\theta:|\T_{\theta,\lambda}^c|\sim j\}$. 

Because $j$ and $\lambda$ are dyadic numbers, there is a choice of $(\lambda,j)$ so that
\[ \|\sum_{R^{-10^3}\le \lambda\le 1}\sum_{1\le j\le R^{22}}\sum_{\theta\in\Theta_j(\lambda)}\sum_{T\in\T_{\theta,\lambda}^c}\s_Tf_\theta\|_{L^6(Q_R)}\lesssim (\log R)^2\|\sum_{\theta\in\Theta_j(\lambda)}\sum_{T\in\T_{\theta,\lambda}^c}\s_Tf_\theta\|_{L^6(Q_R)}. \]

Take $\tilde{\Theta}=\Theta_j(\lambda)$ and for each $\theta\in \tilde{\Theta}$, take $\tilde{\T}_{\theta}=\mathbb{T}_{\theta,\lambda}^c$.

\end{proof}

\subsection{Proof of Theorem~\ref{decoupling}}

\begin{proof} By Proposition \ref{wpd}, we have
\begin{align*}
    \|\sum_\theta f_\theta\|_{L^6(Q_R)}\lesssim (\log R)^3\|\sum_{\theta\in\tilde{\Theta}}\sum_{T\in\tilde{\T}_\theta}\s_Tf_\theta\|_{L^6(Q_R)} +R^{-3}\big(\sum_\theta\|f_\theta\|_{L^6(\R^2)}^2\big)^{1/2}
\end{align*}
where $\#\tilde{\T}_\theta\sim \#\tilde{\T}_{\theta'}$ for all $\theta,\theta'\in\tilde{\Theta}$ and
\begin{equation}\label{eq: A}
\|\s_Tf_\theta\|_{L^\infty(\R^2)} \sim A:= \max_{\theta'\in \tilde{\Theta}}\max_{T'\in \tilde{\T}_{\theta}} \|\s_{T'}f_{\theta'}\|_{L^\infty(\R^2)}. 
\end{equation}

Since the Fourier transform of $\sum_{T\in\tilde{\T}_\theta}\s_Tf_\theta$ is essentially supported in  $ (1 + (\log R)^{-3}) \theta$, there exists a function $f_\theta'$ with Fourier transform supported in $(1 + (\log R)^{-3})\theta$ such that
\[ \sum_{T\in\tilde{\T}_\theta}\s_Tf_\theta(x)=f_\theta'(x)+O(R^{-998})A.\]
Thus 
\begin{align}
    \|\sum_{\theta\in\tilde{\Theta}}\sum_{T\in\tilde{T}_\theta}\s_Tf_\theta\|_{L^6(Q_R)} &\le \|\sum_{\theta\in\tilde{\Theta}}f_\theta'\|_{L^6(Q_R)}+R^{-997}A.
\end{align}
The functions $f_\theta'$ have Fourier support in $(1 + (\log R)^{-3})\theta$. We may split $\tilde{\Theta}$ into $\sim 1 $ sets $\tilde{\Theta}_i$ where for distinct $\theta,\theta'\in \tilde{\Theta}_i$,
\[ (2\theta)\cap (2\theta')=\emptyset. \]
Then for some $i$, 
\[ \|\sum_{\theta\in\tilde{\Theta}}f_\theta'\|_{L^6(Q_R)}^6\lesssim \|\sum_{\theta\in\tilde{\Theta}_i}f_\theta'\|_{L^6(Q_R)}^6 ,\]
and it follows from Theorem \ref{Main} that 
\begin{equation}\label{f'} \|\sum_{\theta\in\tilde{\Theta}}f_\theta'\|_{L^6(Q_R)}^6\lesssim (\log R)^{c}(\sum_{\theta\in\tilde{\Theta}_i}\|f_\theta'\|_{L^\infty(\mathbb{R}^2)}^2)^2\sum_{\theta\in\tilde{\Theta}_i}\|f_\theta'\|_{L^2(\mathbb{R}^2)}^2. \end{equation}
Note that
\[ (\sum_{\theta\in\tilde{\Theta}_i}\|f_\theta'\|_{L^\infty(\mathbb{R}^2)}^2)^{1/2}\le (\sum_{\theta\in\tilde{\Theta}}\|\sum_{T\in\tilde{\T}_\theta}\s_Tf_\theta\|_{L^\infty(\mathbb{R}^2)}^2)^{1/2}+R^{-500}A \]
and
\[\big(\sum_{\theta\in\tilde{\Theta}_i}\|f_\theta'\|_{L^2(\mathbb{R}^2)}^2\big)^{1/2}\lesssim \big(\sum_{\theta\in\tilde{\Theta}}\|\sum_{T\in\tilde{\T}_\theta}\s_Tf_\theta\|_{L^2(\mathbb{R}^2)}^2\big)^{1/2}+R^{-500}A. \]

Combining these observations with (\ref{f'}) gives
\begin{align} \label{here'}
\|\sum_{\theta\in\tilde{\Theta}}f_\theta'\|_{L^6(Q_R)}^6  \lesssim (\log R)^{c} (\sum_{\theta\in\tilde{\Theta}}\|\sum_{T\in\tilde{\T}_\theta}\s_Tf_\theta\|_{L^\infty(\mathbb{R}^2)}^2)^2\sum_{\theta\in\tilde{\Theta}}\|\sum_{T\in\tilde{\T}_\theta}\s_Tf_\theta\|_{L^2(\mathbb{R}^2)}^2 + R^{-2000} A^6.
\end{align} 
The second term is is bounded by 
\begin{equation}\label{eq: tail}
R^{-2000} A^6 \leq R^{-2000} (\sum_{\theta}\|f_{\theta}\|_{L^6(\mathbb{R}^2)}^2)^3
\end{equation}
 using Lemma~\ref{lem: lc} and  the fact that $0\leq \psi_T\leq 1$. 
It remains to analyze the first term in the upper bound in (\ref{here'}). 
For each $\theta\in\tilde{\Theta}$, we have
\[ |\sum_{T\in\tilde{\T}_\theta}\s_Tf_\theta(x)|\lesssim |\sum_{\substack{T\in\tilde{\T}_\theta\\ x\in\log RT}}\s_Tf_\theta(x)|+R^{-1000}A\le (\log R)^2 A+R^{-1000}A.  \]
This leads to the following upper bound for the first term on the right hand side of (\ref{here'}):
\begin{equation}\label{eq: 1}
(\sum_{\theta\in\tilde{\Theta}}\|\sum_{T\in\tilde{\T}_\theta}\s_Tf_\theta\|_{L^\infty(\mathbb{R}^2)}^2)^2 \leq (\log R)^8(\#\tilde{\Theta}A^2)^2 + R^{-1000}A^4.
\end{equation}
For each $\theta\in\tilde{\Theta}$, we also have
\begin{align*}
	\int|\sum_{T\in\tilde{\T}_\theta}\s_Tf_\theta|^2 dx & \lesssim (\log R)^2\sum_{T\in\tilde{\T}_\theta}\int_{\log RT}|\s_Tf_\theta|^2 dx +R^{-1998}A^2    \\
	& \lesssim  \log R)^4 \# \tilde{\Theta} \#\tilde{\T}_{\theta}A^2 |T| + R^{-1998}A^2. 
	\end{align*}
Combining this with \eqref{eq: 1} leads to the upper bound 
\begin{equation}\label{eq: upper}
\|\sum_{\theta\in\tilde{\Theta}}f_\theta'\|_{L^6(Q_R)}^6  \lesssim (\log R)^{c + 12} \#\tilde{\Theta}^3  \#\tilde{\T}_{\theta} A^6 |T| + R^{-1000} A^6.
\end{equation}

%
Finally, note that for each $\theta\in\tilde{\Theta}$, 
\[ \#\tilde{\T}_\theta|T| A^6\lesssim (\log R)^2 \sum_{T\in\tilde{\T}_\theta}\int|\s_Tf_\theta|^6\w_{T}+R^{-2000}A^6\]
where we used the locally constant property. Finally, since the $\ell^6$ norm is bounded by the $\ell^1$ norm and  $ 0 \leq \s_T  \leq 1$,
\[ \sum_{T\in\tilde{\T}_\theta}\int|\s_Tf_\theta|^6\le  \int (\sum_{T\in\tilde{\T}_\theta}  |\s_Tf_\theta|)^6 \leq  \|f_\theta\|_{L^6(\R^2)}^6. \]
In summary, we finish the proof by combining \eqref{eq: upper} with \eqref{eq: A} and 
\[ \#\tilde{\Theta}^3\#\tilde{\T}_\theta|T|A^6 \lesssim (\log R)^2\big(\sum_{\theta\in\tilde{\Theta}}    \|f_\theta\|_{L^6(\R^2)}^2\big)^3.\]

\end{proof}

\vspace*{.2in}

\section*{Appendix}

In the appendix, we show that a slight modification of the proof gives Theorem~\ref{Main} and Theorem~\ref{decoupling} for a set of curves:
$$\mathcal{K}:= \{(\xi_1, h(\xi_1)): |\xi_1|\leq 1\}, $$ where $h$ are  $C^2$ functions  satisfying $h(0)=h'(0)=0$, and  $1/2\leq h''(\xi_1)\leq 2 \text{~for~} |\xi|\leq 1$.  In particular, one can cut the unit circle into $O(1)$ arcs, such that each arc is part of a curve in $\mathcal{K}$ after translation and rotation. The truncated parabola $\mathbb{P}^1$ is also in $\mathcal{K}$.  

 To prove Theorem~\ref{Main}  and Theorem~\ref{decoupling} for the curves in $\mathcal{K}$,  replace the affine map~\eqref{affine map}  in Lemma~\ref{rescaling} by 
\begin{equation}\label{affine map2}
\ell(\xi_1,\xi_2)=\big( \nu (\xi_1-c),\nu^2 (\xi_2 -h(c) -h'(c) (\xi_1-c))\big)
\end{equation}
with $\nu=R_k^{1/2}$. 
Since the curve $\{ \big(\nu(\xi_1 -c), \nu^2 (h(\xi_1) -h(c) -h'(c)(\xi_1-c))\big):   |\xi_1-c|\leq \nu^{-1}\} $ is also in $\mathcal{K}$, the proof of Lemma~\ref{rescaling} remains unchanged provided that Proposition~\ref{mainp} holds for all the curves in $\mathcal{K}$ for a smaller $R$. 

Then it suffices to  check \eqref{cordoba} for the curves in $\mathcal{K}$.
Let $(\xi, h(\xi))$ be the center of $\tau_j$.  Assume that
$$
 \xi, \xi'' \in \tau_k, \,\,\,  \xi',\xi'''\in \tau_k',$$
 and 
$$\xi -\xi'' = \xi'''-\xi' + O(R_j^{-1/2}),$$ then $h(\xi)-h(\xi'')=h'(\xi_1)(\xi-\xi'')$ for some $\xi_1$ between $\xi$ and $\xi''$, and $h(\xi''')-h(\xi')=h'(\xi_2)(\xi'''-\xi')$ for some $\xi_2$ between $\xi'''$ and $\xi'$. The first derivative  $h'(\xi_1) +(\xi_2-\xi_1)/2\leq h'(\xi_2)$ since $1/2\leq h''(\xi)\leq 1$. By $(\xi_2-\xi_1 )\geq \text{dist}(\tau_k, \tau_k')$, we have $$h(\xi)-h(\xi'')- (h(\xi''') -h(\xi'))\gtrsim  \text{dist}(\tau_k, \tau_k') R_{j}^{-1/2}$$  if  $\tau_j, \tau_j', \tau_j'', \tau_j'''$ have pairwise distance $O(R_{j}^{-1/2})$.  So \eqref{cordoba} is verified.

Decoupling for the circle with explicit decoupling constant $(\log R)^c$ has an application on a problem about sums of two squares.  Such problem arises in the study of Laplace eigenfunctions for the standard two dimensional torus.

Let $\Lambda_m$ be the set of Gaussian integers $\lambda=x+\sqrt{-1}y$, $x, y\in \mathbb{Z}$ with norm $\lambda \overline{\lambda}=m$.

\vspace{5pt}
\textbf{Problem.} Give a nontrivial upper bound for the number of solutions
$$\lambda_1+\lambda_2+\lambda_3=\lambda_4+\lambda_5+\lambda_6; ~~\lambda_j\in \Lambda_m.$$

\begin{corollary}\label{lattice on circle}
	For any $\Lambda\subset \Lambda_m$, if $N= |\Lambda|> (\log m)^{(c+6)/\epsilon}$ for some $\epsilon>0$ and the constant $c$ as in Theorem~\ref{Main}, then 
	$$N_m(\Lambda):= \#\{\lambda_1+\lambda_2+\lambda_3=\lambda_4+\lambda_5+\lambda_6, \lambda_j\in \Lambda\} \lesssim N^{3+\epsilon}. $$

\end{corollary}
\begin{proof}
	Consider the function 
	\[
	g(z)= \sum_{\lambda\in \Lambda} e^{2\pi i \frac{\lambda}{\sqrt{m}} \cdot z}
	\]
for $z\in \mathbb{R}^2$.
Let $Q_0=[0,m(\log m)]^2$ and $\{Q\}$ be a tiling of $\mathbb{R}^2$ with translates of $Q_0$. Let $\{\varphi_{Q}\}$ be the Gaussian partition of unity defined as in Definition~\ref{Gaussian}. Define the weight function $$\psi =\sum_{Q:  dist (Q, Q_0)  \leq m (\log m)^2} \varphi_Q.$$ Then $|\psi -1 |\leq m^{-1000}$ on $Q_0$ and $|\psi|\leq m^{-1000}$ outside of $(\log m)^2 Q_0$ and rapidly decays away from it. Moreover, the Fourier transform $\widehat{\psi}$ is essentially supported on $B(0, m^{-1} )$
		
	We apply Theorem~\ref{decoupling}  (for the circle) to  the function $f(z)=g(z) \psi$ with  $R= m$ and $Q_R=[0,m]^2$.  Then $f_{\theta}=e^{2\pi i \lambda\cdot z/\sqrt{m}}$ for the (unique) $\lambda/\sqrt{m}\in \theta$.   Note that $g(z)$ is periodic: $g(z+\sqrt{m} v)=g(z)$ for any $v\in \mathbb{Z}^2$.  Since $|f-g|\leq m^{-1000}$ on $Q_R$,  we have $$N_m(\Lambda) \lesssim (\log m)^{c+6} |\Lambda|^3 $$ where the $(\log m)^6$ comes from $\psi$ a weight function essentially supported in $[0, (\log m)^3m]^2$. 
\end{proof}

This problem was studied by Bombieri and Bourgain in \cite{BB15} using various methods. In particular, they obtained the bound $O(|\Lambda_m|^{3+\epsilon})$ assuming the Riemann hypothesis and the Birch and Swinnerton-Dyer conjecture for the $L$--functions of elliptic curves over $\mathbb{Q}$ and for a random $m$ with $|\Lambda_m|\sim 2^{\omega(m)}$, $\omega(m)\sim \frac{\log m}{A\log \log m}$ for some constant $A$. 
Based on the  Bourgain-Demeter decoupling,  Zane Li obtained in \cite{ZLi} the result of Bombieri and Bourgain unconditionally for all $m$ with  $|\Lambda_m|>\exp ((\log m)^{1-o(1)})$.
Corollary~\ref{lattice on circle} proves the result for a larger range of $|\Lambda|$:  $|\Lambda|> (\log m)^{c/o(1)}$.

In \cite{BD14},  it was conjectured that for any $\Lambda\subset \Lambda_m$, and any $\epsilon>0$, there exists $C_{\epsilon}$ independent of $m$, such that
$$\#\{\lambda_1+\lambda_2+\lambda_3=\lambda_4+\lambda_5+\lambda_6, \lambda_j\in \Lambda\} \leq C_{\epsilon} |\Lambda|^{3+\epsilon}.$$ 
Corollary~\ref{lattice on circle} confirms for $|\Lambda|> (\log m)^{(c+6)/\epsilon}$.

\end{document}

%% file: dfmacros.tex
\newcommand{\R}{\mathbb R}

\newcommand{\Z}{\mathbb Z}
\newcommand{\C}{\mathbb C}
\newcommand{\N}{\mathbb N}
\newcommand{\T}{\mathbb T}

\newcommand{\B}{\mathbb B}

\newcommand{\w}{\omega}
\newcommand{\e}{\varepsilon}
\newcommand{\g}{\gamma}
\newcommand{\p}{\varphi}
\newcommand{\s}{\psi}

\renewcommand{\a}{\alpha}

\newcommand{\mc}{\mathcal}
\newcommand{\mb}{\mathbb}

\def\set4{\mathcal I}
\def\tup14{(1,2,3,4)}

\def\eps{\varepsilon}

\newcommand\vwidehat[1]{\arraycolsep=0pt\relax%
\begin{array}{c}
\stretchto{
  \scaleto{
    \scalerel*[\widthof{\ensuremath{#1}}]{\kern-.5pt\bigwedge\kern-.5pt}
    {\rule[-\textheight/2]{1ex}{\textheight}} 
  }{\textheight} %
}{0.5ex}\\           
#1\\                 
\rule{-1ex}{0ex}
\end{array}
}
\newtheorem*{comm*}{Comment}
\newtheorem{definition}[prop]{Definition}
\newtheorem{notation}[prop]{Notation}
\newtheorem{corollary}[prop]{Corollary}
\newtheorem{thm*}{Theorem}[section]

\usepackage{bbm}   
\usepackage{scalerel,stackengine}
\stackMath
\newcommand\widecheck[1]{%
\savestack{\tmpbox}{\stretchto{%
  \scaleto{%
    \scalerel*[\widthof{\ensuremath{#1}}]{\kern-.6pt\bigwedge\kern-.6pt}%
    {\rule[-\textheight/2]{1ex}{\textheight}}
  }{\textheight}%
}{0.5ex}}%
\stackon[1pt]{#1}{\scalebox{-1}{\tmpbox}}%
}